\documentclass[a4paper]{amsart}

\usepackage{amsmath}
\usepackage{amsfonts}
\usepackage{amssymb}
\usepackage{amsthm}
\usepackage{mathrsfs}
\usepackage{yfonts}
\usepackage{enumitem}
\usepackage{tikz}
\usepackage{tikz-cd}
\usepackage{amsaddr}
\usepackage{comment}

\newtheorem{theorem}{Theorem}[section]

\newtheorem{lemma}[theorem]{Lemma}

\newtheorem{proposition}[theorem]{Proposition}
\newtheorem{corollary}[theorem]{Corollary}

\newtheorem*{remark}{Remark}

\newtheorem{question}[theorem]{Question}
\newtheorem{definition}[theorem]{Definition}

\newtheorem*{acknowledgements}{Acknowledgements}

\newcommand{\existsinfty}{\exists^\infty}

\newcommand{\aaa}{\mathfrak{a}}
\newcommand{\bbb}{\mathfrak{b}}
\newcommand{\ddd}{\mathfrak{d}}
\newcommand{\sss}{\mathfrak{s}}
\newcommand{\rrr}{\mathfrak{r}}

\newcommand{\ppp}{\mathfrak{p}}

\newcommand{\mmm}{\mathfrak{m}}

\newcommand{\continuum}{\mathfrak{c}}

\newcommand{\Meager}{\mathcal{M}}
\newcommand{\Null}{\mathcal{N}}
\newcommand{\E}{\mathcal{E}}

\newcommand{\Baire}{\omega^{\omega}}
\newcommand{\baire}{\omega^{<\omega}}
\newcommand{\Cantor}{2^{\omega}}
\newcommand{\cantor}{2^{<\omega}}

\newcommand{\add}{\mathrm{add}}
\newcommand{\cof}{\mathrm{cof}}
\newcommand{\cov}{\mathrm{cov}}
\newcommand{\non}{\mathrm{non}}

\newcommand{\J}{\mathcal{J}}
\newcommand{\I}{\mathcal{I}}
\newcommand{\RandomGraph}{\mathcal{R}}
\newcommand{\Solecki}{\mathcal{S}}
\newcommand{\ED}{\mathcal{ED}}
\newcommand{\FF}{\mathcal{F}in\times\mathcal{F}in}
\newcommand{\nwd}{nwd}
\newcommand{\FIN}{\mathcal{F}in}
\newcommand{\Z}{\mathcal{Z}}
\newcommand{\SUM}{\mathcal{S}um}

\newcommand{\SPL}{\mathcal{S}pl}
\newcommand{\conv}{conv}
\newcommand{\trN}{\mathrm{tr}(\Null)}
\newcommand{\WANDER}{\mathcal{W}}
\newcommand{\SC}{\mathcal{SC}}
\newcommand{\trE}{\mathrm{tr}(\E)}
\newcommand{\Br}{\mathcal{B}r}
\newcommand{\BI}{\mathcal{BI}}

\newcommand{\IAP}{\mathbb{IAP}}
\newcommand{\ILB}{\mathbb{ILB}}

\title{On antichain numbers of algebras $\mathcal{P}(\omega)/\mathcal{J}$\\ and the splitting ideal}
\author{Aleksander Cieślak}
\address{Faculty of Pure and Applied Mathematics, Wrocław University of Science and Technology, Wybrzeże Stanisława Wyspiańskiego 27, 50-370 Wrocław, Poland}

\begin{document}
\begin{abstract}
    In this article, we study combinatorial properties of a certain ideal on $\omega$, called the \emph{Splitting ideal}. We calculate its cardinal invariants and its position in the Katětov order among other definable ideals. We also study the antichain numbers $\mathfrak{a}(\mathcal{J})$ of algebras $\mathcal{P}(\omega)/\mathcal{J}$ for various Borel ideals. We show that $\textrm{min}\{\mathfrak{b},\textrm{cov}^{+}_{h}(\mathcal{J})\}\leq\mathfrak{a}(\mathcal{J})$ holds for a wide class of ideals, including all $F_{\sigma}$-ideals, all analytic $P$-ideals and many other examples. We also show that $\mathfrak{b}\leq\mathfrak{a}(\mathcal{J})$ holds for \emph{convergent ideal} and for \emph{Boring ideal}. Finally, we will show the consistency of $\mathfrak{a}(\mathcal{J})<\mathfrak{b}$ for the \emph{Van der Waerden's ideal} and the linear growth ideal.
\end{abstract}
\maketitle

\section{Introduction}
Ideals on countable sets play a fundamental role in infinite combinatorics and study of cardinal invariants of the continuum. This article is devoted to two subjects: study of a certain ideal, called the \emph{Splitting ideal} and of \emph{antichain numbers} of the algebras $\mathcal{P}(\omega)/\J$ for various definable ideals. The splitting ideal is an interesting example of an ideal that was defined by Sabok and Zapletal in \cite{SabZapl} in order to characterize adding splitting reals by certain class of idealized forcing notions. We will be interested in the cardinal invariants of the splitting ideal as well as its position in the Katětov order among other known definable ideals. In particular, we will show that, in the Katětov order, the Splitting ideal is above \emph{convergent ideal}, but below the \emph{nowhere-dense ideal}, the \emph{trace of measure zero ideal} and the \emph{Fubini product of ideal of finite sets}. We will also show that $\cov^{*}(\SPL)$ is equal to the reaping number $\rrr$. Here, the $*$-covering number is the smallest number of elements of the ideal that are needed to infinitely intersect any infinite subset of $\omega$. As a result, we will obtain a new upper bound for the covering number of the trace of null ideal, i.e. $\cov^{*}(\trN)\leq\rrr$. This gives more conviction to the open problem whether $\cov^{*}(\Z)\leq\bbb$, where $\Z$ is the asymptotic density zero ideal. Regarding the second subject, over the years, the Van Douwen's diagram of the algebra $\mathcal{P}(\omega)/\J$ for different ideals $\J$ received attention of many researchers. In particular, several people were interested in the antichain numbers of such algebras. For an ideal $\J$ on $\omega$, by $\aaa(\J)$ we mean the smallest size of an uncountable family of $\J$-positive sets, maximal with respect to the property that the intersection of any two of its elements is in $\J$. We will show that $\min\{\bbb,\cov^{+}_{h}(\J)\}$ is a lower bound for $\aaa(\J)$ for a wide class of definable ideals, called \emph{good} ideals. This class includes all $F_{\sigma}$-ideals, all analytic $P$-ideals and many other known examples. The classical result saying that $\bbb\leq\aaa$ can only be partially extended for algebras of the form $\mathcal{P}(\omega)/\J$. For example, $\bbb\leq\aaa(\J)$ holds for all analytic $P$-ideals $\J$ and for ideals like $\FF$ and $\ED$. However, for some ideals, the inequality $\aaa(\J)<\bbb$ is consistent. For example, Steprāns in \cite{Steprans} showed that $\aaa(\nwd)<\bbb$ holds in the Laver's model for Borel conjecture and Brendle in \cite{BrendleAntichains} showed that $\aaa(\ED_{fin})<\bbb$ holds in the Hechler's model. We will show that another examples of ($F_{\sigma}$) ideals for which $\aaa(\J)<\bbb$ is consistent is the Van der Waerden's ideal and the linear growth ideal. Finally, we will show that $\bbb\leq\aaa(\J)$ holds for \emph{convergent ideal} $\conv$ and for \emph{Boring ideal} $\BI$ (studied in \cite{KwelaUnboring} and \cite{BrendleGuzHruRaghavan}).\\

By $\Meager$ we will denote the $\sigma$-ideals of meager subsets of $\Cantor$ with the product topology. By $\Null$ we will denote the $\sigma$-ideals of the Lebesgue measure zero sets in $\Cantor$ with standard product measure. By $\E$ we will denote the $\sigma$-ideal generated by closed sets of measure zero in $\Cantor$. It is easy to see, that $\E\subseteq\Null\cap\Meager$ and the inclusion is proper. If $\I$ is a $\sigma$-ideal on $X$, then the standard cardinal characteristics of $\I$ are
\begin{itemize}
    \item[--] $\mathrm{add}(\I)=\min\{|\mathcal{F}|:\mathcal{F}\subseteq\I$ and $\bigcup\mathcal{F}\notin\I\}$,
    \item[--] $\mathrm{cov}(\I)=\min\{|\mathcal{F}|:\mathcal{F}\subseteq\I$ and $\bigcup\mathcal{F}=X\}$,
    \item[--] $\mathrm{non}(\I)=\min\{|Y|:Y\subseteq X$ and $X\notin\I\}$,
    \item[--] $\mathrm{cof}(\I)=\min\{|\mathcal{F}|:\mathcal{F}\subseteq\I$ and $\forall Y\in\I$ $\exists Y'\in\mathcal{F}$ $Y\subseteq Y'\}$.
\end{itemize}
These are known as \emph{the additivity}, \emph{the covering}, \emph{the uniformity} and \emph{the cofinality} of the ideal $\I$ respectively.
For $f,g\in\omega^{\omega}$ we write $f\leq^{*} g$ if $g$ dominates $f$, i.e. $f(n)\leq g(n)$ for almost all $n\in\omega$. We will call $\mathcal{F}\subseteq\omega^{\omega}$ a \emph{dominating family} if for every $g\in\omega^{\omega}$ there is $f\in\mathcal{F}$ which dominates $g$. We will call $\mathcal{F}\subseteq\omega^{\omega}$ an \emph{unbounded family} if there is no $g\in\omega^{\omega}$ which dominates every $f\in\mathcal{F}$. The bounding and dominating numbers are defined as follows.
\begin{itemize}
    \item[--] $\bbb$ is defined as a minimal cardinality of an unbounded family,
    \item[--] $\ddd$ as a minimal cardinality of an dominating family.
\end{itemize}
Assume that $A,B\in[\omega]^{\omega}$. We will write $A\subseteq^{*}B$ if $A\setminus B$ is finite and $A=^{*}B$ if $A\subseteq^{*}B$ and $B\subseteq^{*}A$, i.e. $A$ and $B$ differ on finite set. We will say that $A$ splits $B$ if both $A\cap B$ and $B\setminus A$ are infinite. A family $\mathcal{A}\subseteq[\omega]^{\omega}$ is a \emph{splitting family} if for every $B\in[\omega]^{\omega}$ there is $A\in\mathcal{A}$ that splits $B$. A family $\mathcal{A}\subseteq[\omega]^{\omega}$ is a \emph{reaping (unsplit) family} if there is no single $B\in[\omega]^{\omega}$ that splits all elements of $\mathcal{A}$. The splitting and reaping numbers are defined as follows. 
\begin{itemize}
    \item[--] $\sss$ is defined as a minimal cardinality of a splitting family,
    \item[--] $\rrr$ as a minimal cardinality of a reaping family.
\end{itemize}
It is well known that inequalities $\sss\leq\ddd,\non(\Null),\non(\Meager)$ and $\cov(\Null),\cov(\Meager),\bbb\leq\rrr$ hold and $\sss$ together with $\rrr$ are mutually incomparable on the basis of ZFC.\\
A family $\mathcal{F}\subseteq[\omega]^{\omega}$ has the finite intersection property (fip. for short) if any intersection of finitely many elements of the family is infinite. A set $A\in[\omega]^{\omega}$ is a pseudo-intersection of a family $\mathcal{F}\subseteq[\omega]^{\omega}$ if $A\subseteq^{*}F$ for any $F\in \mathcal{F}$. We define the pseudo-intersection number as
\begin{itemize}
    \item[--] $\ppp=\min\{|\mathcal{F}|:\mathcal{F}\subseteq[\omega]^{\omega}$ has fip. but no pseudointersection$\}$
\end{itemize}
It is well-known result of Bell, that the Martin's number of $\sigma\mathrm{-centered}$ partial orders $\mmm_{\sigma\mathrm{-centered}}$ is equal $\ppp$. The following diagram summarizes the ZFC-provable inequalities between the cardinal invariants mentioned so far.
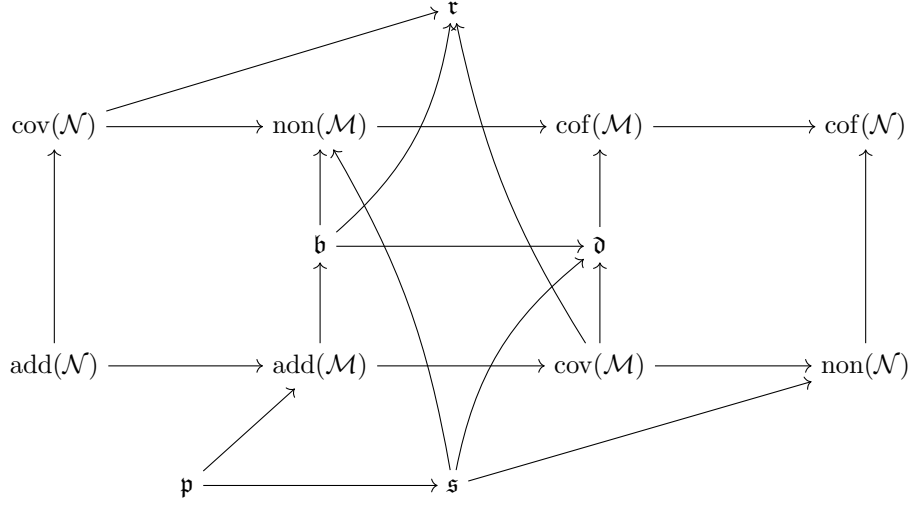
\begin{center}
\begin{tikzpicture}[]
  \matrix[matrix of math nodes,column sep={50pt,between origins},row
    sep={45pt,between origins},nodes={asymmetrical rectangle}] (s)
  {
    &&&|[name=r]| \rrr \\
    |[name=covn]| \mathrm{cov}(\Null) &&|[name=nonm]| \mathrm{non}(\Meager) &&|[name=cofm]| \mathrm{cof}(\Meager) &&|[name=cofn]| \mathrm{cof}(\Null)\\
    &&|[name=b]| \bbb &&|[name=d]| \ddd \\
    |[name=addn]| \mathrm{add}(\Null) &&|[name=addm]| \mathrm{add}(\Meager) &&|[name=covm]| \mathrm{cov}(\Meager) &&|[name=nonn]| \mathrm{non}(\Null) \\
    &|[name=ppp]| \ppp &&|[name=sss]| \sss \\
  };
    \draw[->] 
            (addn) edge (addm)
            (addm) edge (covm)
            (covm) edge (nonn)
            (addm) edge (b)
            (b) edge (d)
            (covm) edge (d)
            (covn) edge (nonm)
            (addn) edge (covn)
            (b) edge (nonm)
            (d) edge (cofm)
            (nonm) edge (cofm)
            (cofm) edge (cofn)
            (nonn) edge (cofn)
            
            (sss) edge [bend right=10] (nonm)
            (sss) edge (nonn)
            (covm) edge [bend left=10] (r)
            (covn) edge (r)
            
            (sss) edge [bend left=20] (d)
            (b) edge [bend right=20] (r)

            (ppp) edge (sss)
            (ppp) edge (addm)
            ;
\end{tikzpicture}
\end{center}
An ideal $\J$ on $\omega$ is tall if for every $A\in[\omega]^{\omega}$ there is $B\in[A]^{\omega}$ such that $B\in\J$. An ideal $\J$ is a $P$-ideal if for every $\{A_{n}:n\in\omega\}\subseteq\J$ there is $B\in\J$ such that $A_{n}\subseteq B$ for every $n\in\omega$. An ideal $\J$ is a $P^{+}$-ideal if for every decreasing sequence $\{A_{n}:n\in\omega\}\subseteq\J^{+}$ there is $B\in\J^{+}$ such that $B\subseteq^{*}A_{n}$ for every $n\in\omega$. As usual, when we refer to the descriptive complexity of an ideal, we mean its complexity as a subset of Cantor space $\Cantor$.\\

Recall the standard cardinal invariants associated with an ideal $\J$:
\begin{itemize}
    \item[--] $\add^*(\J) = \mathrm{min}\{ |\mathcal{F}| : \mathcal{F} \subseteq \J \text{ and } \forall X \in \J$ $\exists F$ $\in \mathcal{F}$ $(F \nsubseteq^* X) \}$,
    \item[--] $\cov^*(\J) = \mathrm{min}\{ |\mathcal{F}| : \mathcal{F} \subseteq \J \text{ and } \forall X\in[\omega]^{\omega}$ $\exists F \in \mathcal{F}$ $(|F \cap X|=\omega) \}$,
    \item[--] $\non^*(\J) = \mathrm{min}\{ |\mathcal{F}| : \mathcal{F} \subseteq [\omega]^{\omega} \text{ and } \forall X\in\J$ $\exists F \in \mathcal{F}$ $(F \cap X=^*\emptyset) \}$,
    \item[--] $\cof^*(\J) = \mathrm{min}\{ |\mathcal{F}| : \mathcal{F} \subseteq \J \text{ and } \forall X\in\J$ $\exists F \in \mathcal{F}$ $(X \subseteq^{*} F) \}$.
\end{itemize}
They are called \emph{the additivity}, \emph{the covering}, \emph{the uniformity} and \emph{the cofinality} of the ideal $\J$ respectively. It is not difficult to see that an ideal $\J$ is a $P$-ideals if and only if $\add^*(\J)>\omega$. Also, $\cov^{*}(\J)$ is not defined if and only if $\J$ is not tall. The following diagram contains well-known ZFC relations between these invariants
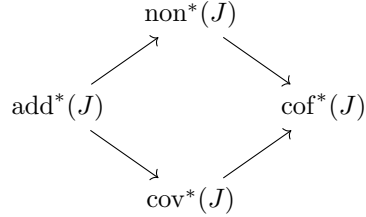
\begin{center}
\begin{tikzpicture}
  \matrix[matrix of math nodes,column sep={50pt,between origins},row
    sep={35pt,between origins},nodes={asymmetrical rectangle}] (s)
  {
    &|[name=non]| \non^{*}(J) &\\
    |[name=add]| \add^{*}(J) &&|[name=cof]| \cof^{*}(J) \\
    &|[name=cov]| \cov^{*}(J) &\\
  };
    \draw[->] 
            (add) edge (cov)
            (add) edge (non)
            (cov) edge (cof)
            (non) edge (cof);
\end{tikzpicture}
\end{center}
As usual, the arrows point towards the greater cardinal invariant. These cardinal invariants have been studied by many researchers (see survey \cite{Hrus}). The interested reader can consult \cite{Hrus} or \cite{Meza}.\\

A powerful tool for classifying ideals on countable sets is the following notion of Katětov reducibility.
\begin{definition}
    Suppose that $\J_{0}$ and $\J_{1}$ are ideals on $\omega$. We will write $\J_{0}\leq_{K}\J_{1}$, and say that $\J_{0}$ is Katětov reducible to $\J_{1}$, if there is a function $\pi:\omega\rightarrow\omega$ such that $\pi^{-1}[A]\in\J_{1}$ for any $A\in\J_{0}$. If such a function is finite-to-one the we call this Katětov-Blass reduction and write $\J_{0}\leq_{KB}\J_{1}$. If $\J_{0}$ is Katětov reducible to $\J_{1}$ and vice versa, we will say that the two ideals are Katětov equivalent.
\end{definition}
The following is well known (see \cite{Hrus})
\begin{proposition}\label{NonKatětov}
    Suppose that $\J_{0}$ and $\J_{1}$ are ideals on $\omega$. Then
    \begin{itemize}
        \item $\J_{0}\leq_{K}\J_{1}$ implies $\mathrm{cov}^*(\J_{1})\leq \mathrm{cov}^*(\J_{0})$,
        \item $\J_{0}\leq_{KB}\J_{1}$ implies $\mathrm{non}^*(\J_{0})\leq \mathrm{non}^*(\J_{1})$.
    \end{itemize}
\end{proposition}
Below we will gather all the ideals related to this work.

\subsection{Examples of $F_{\sigma}$-ideals}\index{Ideal!$F_{\sigma}$-ideal}
A function $\phi:\mathcal{P}(\omega)\rightarrow[0,\infty]$ is a lower semi-continuous submeasure (l.s.c.s.m. for short) if the following conditions hold:
\begin{itemize}
    \item[--] $\phi(\emptyset)=0$, $\phi(\{n\})<\infty$ for every $n\in\omega$,
    \item[--]   $\phi(X)\leq\phi(Y)$ for $X\subseteq Y$,
    \item[--]  $\phi(X\cup Y)\leq\phi(X)+\phi(Y)$ for all $X, Y\subseteq\omega$,
    \item[--] $\phi(X)=\sup\{\phi(X\cap n):n\in\omega\}$ for each $X\subseteq\omega$.
\end{itemize}
There is a natural way to associate an ideal to such $\phi$. Define 
\begin{center}
    $Fin(\phi)=\{X\subseteq\omega:\phi(X)<\infty\}$
\end{center}
It is easy to see that this is an $F_{\sigma}$ $P^{+}$-ideal. The following is a well-known characterization of $F_{\sigma}$-ideals.
\begin{theorem}(Mazur,\cite{Mazur})\label{Mazur}
    An ideal $\J$ is $F_{\sigma}$ if and only if $\J=Fin(\phi)$ for some l.s.c.s.m. $\phi$.
\end{theorem}

\textbf{Summable ideal}:\\
        Given a function $f:\omega\rightarrow[0,+\infty)$ such that $\Sigma_{n\in\omega}f(n)=\infty$ we define a summable ideal corresponding to $f$ as follows:
        \begin{center}
            $\mathcal{I}_{f}=\{A\subseteq\omega:\Sigma_{n\in A}f(n)<\infty\}$
        \end{center}
        Clearly, $\mathcal{I}_{f}$ is an $F_{\sigma}$-ideal with the associated submeasure defined as $\phi(A)=\Sigma_{n\in A}f(n)$. It follows that $\mathcal{I}_{f}$ is a $P^{+}$-ideal. It is also a $P$-ideal. $\mathcal{I}_{f}$ is tall if and only if $lim_{n}f(n)=0$. In this article, we will only consider the classical summable ideal $\SUM$ defined as $\mathcal{I}_{f}$ where $f(n)=\frac{1}{n+1}$.\\
        
\textbf{Eventually different ideal}:
        Let $\ED$ be the ideal on $\omega\times\omega$ consisting of these sets $A$ such that for some $N$ for all $n>N$ the set $\{m: (n,m)\in A\}$ has at most $N$ elements. Let $\ED_{fin}$ be the restriction $\ED|_{\Delta}$ where $\Delta=\{(n,m)\in\omega\times\omega:m\leq n\}$. For $F\in\Baire$ one could consider an ideal $\ED_{F}$ defined as restriction of $\ED$ to the set of form $\{(n,m)\in\omega\times\omega:m\leq F(m)\}$. It is well-known that such $\ED_{F}$ is Katětov equivalent with $\ED_{fin}$. It follows that the two have the same covering and uniformity numbers.\\
        Both $\ED$ and $\ED_{fin}$ are $F_{\sigma}$-ideals. It follows that both $\ED$ and $\ED_{fin}$ are $P^+$-ideals but neither $\ED$ nor $\ED_{fin}$ is a $P$-ideal. It is easy to see that in the definition of $\ED_{fin}$, if we replace the diagonal by any increasing function, we get a Katětov equivalent ideal. It is also easy to see that every $\ED$-positive set contains $\ED$-positive set $F$ such that $\ED|_{F}$ is Katětov equivalent to $\ED_{fin}$.\\
        The cardinal invariants $\cov^{*}(\ED_{fin})$ and $\non^{*}(\ED_{fin})$ are interesting invariants in the sense that they allow to express $\non(\Meager)$ and $\cov(\Meager)$ as a maximum as minimum.
    \begin{theorem}(Hrušák, Meza-Alcántara, Minami; \cite{HrusakMeza})
        The following (in)equalities holds:
        \begin{enumerate}
            \item[--] $\cov^{*}(\ED)=\non(\Meager)$ and $\non^{*}(\ED)=\omega$,
            \item[--] $\cov(\Null)\leq\cov^{*}(\ED_{fin})\leq\non(\Meager)$,
            \item[--] $\cov(\Meager)\leq\non^{*}(\ED_{fin})\leq\non(\Null)$
        \end{enumerate}
        also
        \begin{enumerate}
            \item[--] $\non(\Meager)=\max\{\bbb,\cov^{*}(\ED_{fin})\}$,
            \item[--] $\cov(\Meager)=\min\{\ddd,\non^{*}(\ED_{fin})\}$.
        \end{enumerate}
    \end{theorem}
\textbf{Solecki ideal}:\\
        Let $\Solecki$ be the Solecki's ideal, the ideal on $\Omega=:\{C\subseteq\Cantor:C$ is a clopen set and $\lambda(C)=1/2\}$ generated by sets of the form $[F]=:\{C\in\Omega: C\cap F\neq\emptyset\}$ where $F\in[2^{\omega}]^{<\omega}$. The ideal $\Solecki$ is a tall $F_{\sigma}$-ideal: for an $A\subseteq\Omega$, the submeasure $\phi(A)$ measures how many elements of $\Cantor$ do we need to intersects every element of $A$. It follows that $\Solecki$ is a $P^{+}$-ideal. It is not a $P$-ideal.\\
        
\textbf{Van der Waerden ideal}:\\
        Let $\WANDER$ be the ideal on $\omega$ consisting of these $A\subseteq\omega$ which does not contain an arithmetical progression of arbitrary length. The ideal $\WANDER$ is tall but not a $P$-ideal. The $\WANDER$ ideal is an $F_{\sigma}$-ideal: for $A\subseteq\omega$, the submeasure $\phi(A)$ measures how long arithmetical sequence does $A$ contain. It follows that $\WANDER$ is a $P^+$-ideal. Van der Waerden’s well-known Theorem says that $\WANDER$ is a proper ideal. Szemerédi’s famous Theorem says that $\WANDER\subseteq\Z$. The stronger statement $\WANDER\subseteq\SUM$ is still open Erdős prize problem.\\
        
\textbf{Random Graph ideal}:\\
        Let $(\omega,E)$, $E\subseteq[\omega]^{2}$ be the random graph, i.e. the only graph (up to isomorphism) with the following property: for every pair $A,B\subseteq\omega$ of non-empty, finite, disjoint sets there is an $n\in\omega\setminus(A\cup B)$ such that $\{\{n,a\}:a\in A\}\subseteq E$ and $\{\{n,b\}:b\in b\}\cap E=\emptyset$. The random graph ideal $\RandomGraph$ is the ideal generated by homogeneous subsets (cliques and anticliques) of the random graph. The ideal $\RandomGraph$ is tall and it is an $F_{\sigma}$-ideal: for $A\subseteq E$, the submeasure $\phi(A)$ measures what is the minimum amount of cliques and anticliques in $E$ needed to cover $A$. It follows that $\RandomGraph$ is a $P^{+}$-ideal but not a $P$-ideal.

\subsection{Examples of analytic $P$-ideals}
The following is a well-known characterization of analytic $P$-ideals using lower semi continuous submeasures.
\begin{theorem}(Solecki,\cite{SoleckiExh})\index{Ideal!analytic $P$-ideal}
    Every analytic $P$-ideal is of the form
    \begin{center}
        $Exh(\phi)=\{X\subseteq\omega: lim_{n}\phi(X\setminus n)=0\}$
    \end{center}
    for some lower semi continuous submeasure $\phi$.
\end{theorem}
It is not difficult to see that $Exh(\phi)$ is in fact an $F_{\sigma\delta}$-ideal. It is well-known, that for a l.s.c.s.m. $\phi$, the ideal $Exh(\phi)$ is tall if and only if $lim_{n}\phi(\{n\})=0$. Our first example of analytic $P$-ideal is the already defined summable ideal $\SUM$. In this article we will be interested in the following two analytic P-ideals.\\

\textbf{Density zero ideal}:\\
        The density zero ideal $\Z$ consists of these $A\subseteq\omega$ such that $lim_{n\to\infty}\frac{|A\cap n|}{n}=0$. This ideal is tall and can be expressed in the form of $Exh(\phi)$ for l.s.c.s.m. $\phi$ defined as
        \begin{center}
            $\phi(A)=:sup\{\frac{|A\cap [2^{n},2^{n+1})|}{2^{n}}:n\in\omega\}$
        \end{center}
        It is known that the density zero ideal is Katětov equivalent with the ideal of subsets $A\subseteq\cantor$ such that $lim_{n\to\infty}\frac{|A\cap 2^{n}|}{2^{n}}=0$ (see \cite{Hernand}, Theorem 3.4).\\
        
\textbf{Trace of null ideal}:\\
        The trace of null ideal $\trN$ is an ideal on $\cantor$ consisting of those sets $A\subseteq\cantor$ such that the set $G(A):=\{x\in\Cantor:\existsinfty_{n} x|_{n}\in A\}$ is of Lebesgue measure zero. This is a tall analytic $P$-ideal with a l.s.c.s.m given for $A\subseteq\cantor$ by
        \begin{center}
            $\phi(A)=\Sigma\{\frac{1}{2^{|\sigma|}}:\sigma\in A$ is $\subseteq$-minimal in $A\}$
        \end{center}
        Similarly we define the trace of closed measure zero ideal $\trE$ (even though it is not a $P$-ideal). Clearly, as $\E\subseteq\Null$, we have $\trE\subseteq\trN$.
        
\subsection{Other examples}
Other than $F_{\sigma}$ and analytic $P$-ideals, we will be interested in the following examples.\\

\textbf{Nowhere dense ideal}:\\
        Usually, the ideal $\nwd$ consists of all $A\subseteq\mathbb{Q}\cap[0,1]$ such that $\overline{A}$ is nowhere dense in $[0,1]$. However, there is a different (but Katětov equivalent) presentation of this ideal which we will use in this article. Let $\nwd$ be the ideal on $\cantor$ consisting of these sets $A$ that are nowhere dense in $\cantor$, i.e. for any $\sigma\in\cantor$ there is $\tau\in\cantor$, $\sigma\subseteq\tau$ with $\tau\notin A$. By the Lemma 81 of \cite{GuzmanPhd} the two versions of $\nwd$ ideal defined above are Katětov equivalent. The ideal $\nwd$ is not a $P$-ideal and not a $P^+$-ideal. Its complexity is $F_{\sigma\delta}$.\\
        
\textbf{Fubini product of fin and Boring ideal}:
        Let $\FF$ be ideal on $\omega\times\omega$ consisting of these sets $A$ such that for almost all $n\in\omega$ the set $\{m: (n,m)\in A\}$ is finite. $\FF$ is tall but not a $P$-ideal and not a $P^{+}$-ideal. Its descriptive complexity is $F_{\sigma\delta\sigma}$.\\
        By $\BI$ we will denote the \emph{boring ideal} defined in \cite{KwelaUnboring}. The ideal $\BI$ consists of these $A\subseteq\omega^{3}$ such that for every $n\in\omega$ $A\cap(\{n\}\times\omega^{2})\in\FF$ and for almost all $n\in\omega$ the set $A\cap(\{n\}\times\omega^{2})$ is finite. This ideal appeared also in \cite{BrendleGuzHruRaghavan} in the context of weakly-tight m.a.d. families and was denoted there by $\mathcal{WT}$.
        It is known that $\conv\leq_{K}\BI\leq_{K}\FF$.

\textbf{Convergent ideal}:\\
        The convergence ideal $\conv$ is defined as the ideal on $\mathbb{Q}\cap[0,1]$ generated by convergent sequences. Its descriptive complexity is $F_{\sigma\delta\sigma}$ and it is a tall ideal. It is not a $P$-ideal and not a $P^{+}$-ideal. It is not difficult to see that for every $\conv$-positive $A$ there is $\conv$-positive $B\subseteq A$ such that $\conv|_{B}$ is Katětov equivalent to $\FF$.\\
        We will later need the following two facts regarding the $\conv$ ideal.
        \begin{lemma}(Meza-Alcántara, \cite{Meza})\label{convKatětov}
            The following are equivalent:
            \begin{itemize}
                \item[--] $\conv\leq_{K}\J$,
                \item[--] there is $\{X_{n}:n\in\omega\}\subseteq[\omega]^{\omega}$ such that for every $Y\in\J^{+}$ there is $n\in\omega$ with $|Y\cap X_{n}|=|Y\setminus X_{n}|=\omega$.
            \end{itemize}
        \end{lemma}
        \begin{lemma}(Meza-Alcántara, \cite{Meza})\label{NoFsigmaABOVEconv}
        If $\conv\leq_{K}\J$, then $\J$ cannot be a $P^{+}$-ideal. In particular, no $F_{\sigma}$-ideal is Katětov above $\conv$.
    \end{lemma}
        
\textbf{Ideals generated by m.a.d. families}:\\
        If $\mathcal{A}\subseteq[\omega]^{\omega}$ is an almost disjoint family, then let $\J(\mathcal{A})$ be the ideal generated by $\mathcal{A}$. It is easy to see that $\J(\mathcal{A})$ is tall if and only if $\mathcal{A}$ is maximal almost disjoint family. Ideals of this form cannot be analytic by a well-known result of Mathias \cite{Mathias}, and are not $P$-ideals. It is a folklore result that for every m.a.d. family $\mathcal{A}$ we have $\mathcal{I}(\mathcal{A})\leq_{K}\FF$.
        
The Katětov reductions between most of the ideals listed above can be summarized with the following diagram.
\begin{center}
\begin{tikzpicture}[]
  \matrix[matrix of math nodes,column sep={36pt,between origins},row
    sep={30pt,between origins},nodes={asymmetrical rectangle}] 
  {
    &&|[name=nwd]| \nwd &&|[name=finfin]| \FF &&|[name=Z]| \Z \\
    &&&&&&|[name=trN]| \trN \\
    &&&&|[name=BI]| \BI &&|[name=sum]| \SUM &|[name=WANDER]| \WANDER \\
    &&&&&&|[name=edfin]| \ED_{\mathrm{Fin}} \\
    &&|[name=sol]| \Solecki &&|[name=conv]| conv &&|[name=ed]| \ED  \\
    &&&&|[name=rand]| \RandomGraph \\
    &&&&|[name=fin]| \FIN \\
  };
    \draw[->] 
            (conv) edge (BI)
            (BI) edge (finfin)
            (ed) edge (finfin)
            (conv) edge (nwd)
            (conv) edge (Z)
            (rand) edge (ed)
            (ed) edge (edfin)
            (edfin) edge (sum)
            (rand) edge (conv)
            (fin) edge (rand)
            (fin) edge (sol)
            (sol) edge (nwd)
            (sum) edge (trN)
            (trN) edge (Z)
            (WANDER) edge (Z)
            (edfin) edge (WANDER)
            ;
            
\end{tikzpicture}
\end{center}

\section{Splitting ideal}
\begin{definition}(Sabok and Zapletal, \cite{SabZapl})\index{Splitting ideal}\index{Ideal!$\SPL$}
    Let $\SPL$ be the splitting ideal, an ideal on $\cantor$ generated by the sets of form $S(A)$ for $A\subseteq\omega$, where $S(A)$ consists of those $\sigma\in\cantor$ such that $\sigma|_{A}$ is constant.
    \end{definition}
    The ideal $\SPL$ was introduced by Sabok and Zapletal in \cite{SabZapl} in the context of characterizing when an idealized forcing notion $\mathbb{P}_{\I}:=(\mathcal{B}or\setminus\I,\subseteq)$ adds a splitting (independent) real. Note that for every $A\in[\omega]^{\omega}$ the set $S(A)$ forms a tree in $\cantor$ and that $A\subseteq B$ implies $S(B)\subseteq S(A)$. We will start by checking the tallness of $\SPL$.
    \begin{proposition}
        $\SPL$ is tall.
    \end{proposition}
    \begin{proof}
        Let $\{\sigma_{n}:n\in\omega\}$ be an infinite subset of $2^{<\omega}$. By passing to a subsequence we may assume that $\{\sigma_{n}:n\in\omega\}$ forms either chain of antichain in the order $(2^{<\omega},\subseteq)$. If $\{\sigma_{n}:n\in\omega\}$ is a chain, then for $A=\bigcup_{n}\sigma_{n}\in\Cantor\approx[\omega]^{\omega}$ we have $\{\sigma_{n}:n\in\omega\}\subseteq S(A)$. Assume then that $\{\sigma_{n}:n\in\omega\}$ is an antichain. By Konig's lemma, let $y\in\Cantor$ be an infinite branch in the tree generated by $\sigma_{n}$'s. Let $a_{n}\in\omega$ be defined as $\min\{i\in|\sigma_{n}|:\sigma_{n}(i)\neq y(i)\}$. The $a_{n}$'s are well defined as $\{\sigma_{n}:n\in\omega\}$ is an antichain. Again, by passing to a subsequence we may assume that for each $n\in\omega$ we have that $|\sigma_{n}|<a_{n+1}<|\sigma_{n+1}|$. Let $A\in[\omega]^{\omega}$ be such that $|A\cap[|\sigma_{n}|,a_{n+1})|\leq 1$ for each $n\in\omega$ and $y|_{A}$ is constant. Clearly then $\{\sigma_{n}:n\in\omega\}\subseteq S(A)$.
    \end{proof}
    
\begin{question}
    What do canonical $\SPL$-positive sets look like? Is $\SPL$ homogeneous?
\end{question}
Next, we will focus of the position of $\SPL$ in the Katětov order.
\begin{proposition}\label{SPLKatětovUpper}
    The $\SPL$ ideal is Katětov below $\trE$ and $\FF$.
\end{proposition}
\begin{proof}
    To show $\SPL\leq_{K}\FF$ let for each $n\in\omega$ $x_{n}\in\Cantor$ be the binary sequence starting with $n$-th many 1's followed by 0's only. Let $\pi:\omega\times\omega\rightarrow \cantor$ be defined as $\pi(n,m)=x_{n}|_{m}$. To show that $\pi$ is a Katětov reduction, assume that $A\in[\omega]^{\omega}$. Let $A=\{a_{n}:n\in\omega\}$ and let $f\in\Baire$ be such that $f|_{[a_{n},a_{n+1})}$ is constant with value $a_{n+1}$. To finish the proof note that
    \begin{center}
        $\pi^{-1}[S(A)]\subseteq\{(n,m)\in\omega\times\omega: n\leq a_{0}$ or $m\leq f(n)\}$.
    \end{center}
    To show that $\SPL\leq_{K} \trE$ let $\pi$ be the identity map. Observe that for each $A\in[\omega]^{\omega}$ and $n\in\omega$ we have that 
    \begin{center}
        $\frac{|S(A)]\cap 2^{n}|}{2^{n}}=\frac{1}{2^{|a\cap n|}}$
    \end{center}
    It follows that $G(S(A))=[S(A)]\in \E$. 
    \end{proof}
As $\E\subseteq\Meager\cap\Null$ the inclusion $\trE\subseteq\nwd\cap\trN$ holds and we obtain the following corollary.
\begin{corollary}
    $\SPL\leq_{K}\nwd$ and $\SPL\leq_{K}\trN$.
\end{corollary}
Regarding the ideal $\trE$ Hru\v{s}\'{a}k and Zapletal showed in \cite{HrusakZapletal} that $\cov(\E)\leq\cov^{*}(\trE)\leq\max\{\ddd,\cov(\E)\}$. I do not know if $\cov^{*}(\trE)$ is not simply equal $\cov(\E)$. Notice, for comparison, that for the ideal $tr(\Meager)$, which is Katětov equivalent to $\nwd$, we have $\cov^{*}(\nwd)=\cov(\Meager)$. On the other hand for the $\trN$ ideal, the inequality $\cov(\Null)<\cov^{*}(\trN)$ is consistent. This raises the following question.
\begin{question}
    Is it consistent that $\cov(\E)<\cov^{*}(\trE)$?
\end{question}
Next, we will show that $\SPL$ is Katětov above the convergent ideal. To do this, we will need the following lemma.
\begin{lemma}\label{offBranchSpl}
    If $X\subseteq\cantor$ is such that for some $y\in\Cantor$ we have that:
    \begin{center}
        $\forall n\in\omega$ $\forall i\in 2$, if $i\neq y(n)$, then $\{\sigma\in X:y|_{n}^{\frown}i\subseteq\sigma\}$ is finite
    \end{center}
    then $X\in\SPL$.
\end{lemma}
\begin{proof}
    For $j\in\omega$ let 
    \begin{center}
        $C_{j}=\{\sigma\in X: \sigma\subseteq y|_{j}$ or $( \sigma|_{j}=y|_{j}$ and $\sigma|_{j+1}\neq y|_{j+1})\}$.
    \end{center}
    By the assumption, each $C_{j}$ is finite. Let $i\in 2$ be such that $y(n)=i$ for infinitely many $n$'s. Inductively we construct two infinite sets $A=\{a_{n}:n\in\omega\}$ and $K=\{k_{n}:n\in\omega\}$ such that for every $n\in\omega$ we have:
    \begin{itemize}
        \item[--] $y(a_{n})=y(k_{n})=i$,
        \item[--] $k_{n}<a_{n}<k_{n+1}$,
        \item[--] $k_{n+1}>\max\{|\sigma|:\sigma\in\bigcup\{C_{j}:j\in[k_{n},a_{n})\}\}$,
        \item[--] $a_{n+1}>\max\{|\sigma|:\sigma\in\bigcup\{C_{j}:j\in[a_{n},k_{n+1})\}\}$.
    \end{itemize}
    We claim that $X\subseteq S(A)\cup S(K)$. Let $\sigma\in X$ and let $j\in\omega$ be such that $\sigma|_{j}=y|_{j}$ but $\sigma|_{j+1}\neq y|_{j+1}$. Let $n\in\omega$ be such that $j\in[k_{n},a_{n})$ or $i\in[a_{n},k_{n+1})$. In the first case, we have that $\sigma\in S(K)$, and in the second that $\sigma\in S(A)$. This finishes the proof.
\end{proof}
 We will now show that $\SPL$ is Katětov above $\conv$ ideal.
\begin{proposition}\label{SPLconv}
    $\conv\leq_{K}\SPL$
\end{proposition}
\begin{proof}
    To show this, we will use lemma \ref{convKatětov} by which it is enough to construct countably many infinite subsets of $\cantor$ which split all $\SPL$-positive sets. For $\sigma\in\cantor$ let $X_{\sigma}=\{\tau\in\cantor:\sigma\subseteq\tau\}$. Assume that $X\in\SPL^{+}$ is not splitted by any $X_{\sigma}$. It is easy to see that there is $y\in\Cantor$ such that for every $n\in\omega$ and every $i\in 2$, if $i\neq y(n)$, then $X_{y|_{n}\textbf{}^{\frown}i}\cap X$ is finite. By the lemma \ref{offBranchSpl}, $X$ cannot be $\SPL$-positive. A contradiction.
\end{proof}
Using the lemma \ref{NoFsigmaABOVEconv} and the Proposition \ref{SPLconv} we see that $\SPL$ cannot be $P^{+}$-ideal and thus cannot be Katětov below any $F_{\sigma}$-ideal. In particular, it is not true that $\SPL\leq_{K}\mathcal{I}_{\frac{1}{n}}$ or $\SPL\leq_{K}\mathcal{W}$. I do not know what is the relation between ideals $\BI$ and $\SPL$. We have the following.
\begin{question}
    Is $\SPL\leq_{K}\BI$?
\end{question}
The following diagram shows the position of the ideal $\SPL$ in the Katětov order. 
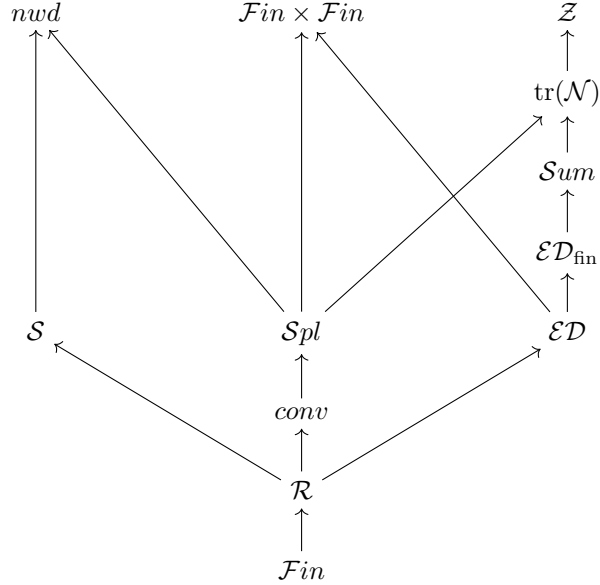
\begin{center}
\begin{tikzpicture}[]
  \matrix[matrix of math nodes,column sep={100pt,between origins},row
    sep={30pt,between origins},nodes={asymmetrical rectangle}] 
  {
    |[name=nwd]| \nwd &|[name=finfin]| \FF &|[name=Z]| \Z \\
    &&|[name=trN]| \trN \\
    &&|[name=sum]| \SUM \\
    &&|[name=edfin]| \ED_{\mathrm{fin}} \\
    |[name=sol]| \Solecki &|[name=SPL]| \SPL &|[name=ed]| \ED  \\
    &|[name=conv]| conv \\
    &|[name=rand]| \RandomGraph \\
    &|[name=fin]| \FIN \\
  };
    \draw[->] 
            (SPL) edge (finfin)
            (ed) edge (finfin)
            (sum) edge (trN)
            (trN) edge (Z)
            (SPL) edge (nwd)
            (SPL) edge (trN)
            (conv) edge (SPL)
            (rand) edge (ed)
            (ed) edge (edfin)
            (edfin) edge (sum)
            (rand) edge (conv)
            (fin) edge (rand)
            (rand) edge (sol)
            (sol) edge (nwd);
            
\end{tikzpicture}
\end{center}

As $\SPL$ is Katětov above $\conv$ it cannot be an $F_{\sigma}$-ideal. We will prove a slightly stronger result, namely that $\SPL$ is not a $\Pi_{0}^{4}$-ideal, using the following result.
\begin{lemma}(Kwela, \cite{KwelaUnboring} Proposition 4.9)
    Suppose that $\J$ is such that there is $A\in\J^{+}$ and a bijection $\pi:A\rightarrow\omega\time\omega$ such that $B\subseteq A$ is in $\J$ if and only if $\pi[B]\in\FF$. Then is not a $\Pi_{0}^{4}$-ideal.
\end{lemma}
\begin{proposition}
    $\SPL$ is not a $\Pi_{0}^{4}$-ideal.
\end{proposition}
\begin{proof}   
    We would like to use the above lemma. Notice that if $\phi:\omega\times\omega\rightarrow\cantor$ is the witness for $\SPL\leq_{K}\FF$ constructed in the Proposition \ref{SPLKatětovUpper}, then the set $X=\phi[\omega\times\omega]$ is $\SPL$-positive and is disjoint union of the sets $X_{n}=\phi[\{n\}\times\omega]\in\SPL$. On the other hand, if $Y\subseteq X$ has a finite intersection with each $X_{n}$, then by the lemma \ref{offBranchSpl} it must belong to $\SPL$.
\end{proof}
The question about exact complexity of the $\SPL$ ideal is open. Clearly, by definition, it is an analytic ideal, but I do not know whether it is analytic complete?
\begin{question}(Sabok, Zapletal; \cite{SabZapl})
    Is the ideal $\SPL$ Borel?
\end{question}
We will now focus on the cardinal invariants of the splitting ideal.
\begin{proposition}
    $\add^{*}(\SPL)=\non^{*}(\SPL)=\omega$. In particular, the splitting ideal $\SPL$ is not a $P$-ideal.
\end{proposition}
\begin{proof}
    For each $\sigma\in \cantor$, let $x_{\sigma}$ be the element of $\Cantor$ that starts with $\sigma$ and continues with zeros. Let $X_{\sigma}=\{x_{\sigma}|_{n}:n\in\omega\}$. We claim that $\{X_{\sigma}:\sigma\in \cantor\}$ witnesses for $\non^{*}(\SPL)$. To show this, let $A\in[\omega]^{\omega}$. Let then $\sigma\in \cantor$ be such that $A\cap|\sigma|$ consists of two elements and let the values of $\sigma$ on these two elements are different. Then, clearly $X_{\sigma}\cap S(A)$ is finite.
\end{proof}
To calculate the covering number of $\SPL$ we will need the following lemma about certain strengthening of reaping families. Call a family $\mathcal{R}\subseteq[\omega]^{\omega}$ \emph{hereditarily reaping} if it is reaping and it is reaping inside of every its element, i.e. for every $R\in\mathcal{R}$ and every $A\in[R]^{\omega}$ there is $R'\in\mathcal{R}$, $R'\subseteq A$ such that either $R'\subseteq^{*} A$ or $A\cap R'=^{*}\emptyset$. The following lemma is known.
\begin{lemma}
    There is a hereditarily reaping family of size $\rrr$.
\end{lemma}
\begin{proof}
    Fix a reaping family $\mathcal{R}$ of size $\rrr$. We will inductively construct an increasing sequence $\{\mathcal{R}_{n}:n\in\omega\}$ with $|\mathcal{R}_{n}|=\rrr$ such that $\mathcal{R}':=\bigcup_{n}\mathcal{R}_{n}$ is hereditarily reaping. Let $\mathcal{R}_{0}=\mathcal{R}$ and to construct $\mathcal{R}_{n+1}$ out of $\mathcal{R}_{n}$ we do as follows: for each $R\in\mathcal{R}_{n}$ use any bijection between $R$ and $\omega$ to place a copy of $\mathcal{R}$ inside of $R$. Declare then $\mathcal{R}_{n+1}$ to be $\mathcal{R}_{n}$ with all such copies for all $R\in\mathcal{R}_{n}$. Clearly $\mathcal{R}'$ is hereditarily reaping.
\end{proof}
\begin{proposition}\label{covSplIsRRR}
    $\cov^{*}(\SPL)=\rrr$
\end{proposition}
\begin{proof}
    To show $\rrr\leq \cov^{*}(\SPL)$ assume that  $\{A_{\alpha}\in[\omega]^{\omega}:$ and $\kappa<\rrr\}$ is given. Let $B$ be such an infinite subset of $\omega$ that splits each $A_{\alpha}$. Define $y\in\Cantor$ such that $y(n)=1$ if and only if $n\in B$. Then note that $Y=\{y|_{n}:n\in\omega\}\subseteq \cantor$ is almost disjoint from each $S(A_{\alpha})$.\\
    To show that $\cov^{*}(\SPL)\leq\rrr$ let $\mathcal{R}$ be a hereditary reaping family of size $\rrr$. We may assume that $\mathcal{R}$ is closed under finite modifications. We claim that the family $\{S(R):R\in\mathcal{R}\}$ witnesses for $\cov^{*}(\SPL)$. Take arbitrary $Y\in[\cantor]^{\omega}$. Consider the following two possibilities. If $Y$ contains $\subseteq$-chain $\{\sigma_{n}:n\in\omega\}$ there must be an $R\in\mathcal{R}$ such that for $y=\bigcup_{n}\sigma_{n}\in\Cantor$ we have $R\cap y^{-1}[1]=\emptyset$ or $R\subseteq y^{-1}[1]$. Clearly in any of the two cases we have $\{\sigma_{n}:n\in\omega\}\subseteq S(R)$. Assume then that $Y$ contains no $\subseteq$-chain. Shrinking $Y$ if necessary, we may assume that $Y$ forms an antichain in $(\cantor,\subseteq)$ and the tree generated by $Y$ contains exactly one infinite branch $y\in\Cantor$. Let $\{\sigma_{n}:n\in\omega\}$ be an length increasing enumeration of the set $Y$ and let $a_{n}\in\omega$ be the smallest $m\in\omega$ such that $\sigma_{n}|_{m}\neq y|_{m}$. Shrinking $Y$ once more we may assume that $a_{n}<|\sigma_{n}|<a_{n+1}$ for each $n\in\omega$. Let then $I_{n}$ be the interval $[a_{n}, |\sigma_{n}|)$. Pick $R_{0}\in \mathcal{R}$ such that $R_{0}\subseteq y^{-1}[0]$ or $R_{0}\subseteq y^{-1}[1]$. Next, using hereditary reaping family twice we may find $R_{1}\in\mathcal{R}$, $R_{1}\subset R_{0}$ such that one of the following three possibilities holds:
    \begin{center}
        $R_{1}\cap \bigcup_{n}I_{n}=\emptyset$ or $R_{1}\subseteq \bigcup_{n}I_{2n}$ or $R_{1}\subseteq \bigcup_{n}I_{2n+1}$
    \end{center}
    Note, that in the first case, we have $\{\sigma_{n}:n\in\omega\}\subseteq S(R_{1})$ and in the second and third case, we have $\{\sigma_{2n+1}:n\in\omega\}\subseteq S(R_{1})$ or $\{\sigma_{2n}:n\in\omega\}\subseteq S(R_{1})$ respectively. It follows that $Y\cap S(R_{1})$ is infinite which finishes the proof.
\end{proof}
Many researchers thought about upper bounds of the covering number of the density zero ideal. Let us gather here several known related results.
\begin{theorem}
    The following inequalities hold in ZFC:
        \begin{itemize}
            \item[--] $\cov^{*}(\Z)\leq \cov^{*}(\trN)\leq\cov^{*}(\SUM)\leq \cov^{*}(\ED_{fin})\leq \non(\Meager)$, (folklore)
            \item[--] $\cov^{*}(\Z)\leq \max\{\cov(\Null),\ddd\}$  (Hrušák and Hern\'{a}ndez-Hern\'{a}ndez; \cite{Hernand}),
            \item[--] $\cov^{*}(\Z)\leq\ddd$ (Raghavan and Shelah; \cite{RaghavanShelah}),
            \item[--] $\cov^{*}(\Z)\leq\cov(\E)$ (Cieślak, Gappo, Martínez, Yamazoe; \cite{CIESGAPPOYAMAZOEMARTINEZ}),
            \item[--] $\cov^{*}(\Z)\leq\max\{\bbb,\sss(\mathfrak{pr})\}$\footnote{$\sss(\mathfrak{pr})$ is a variant of splitting number that will not be relevant for this article; for the definition of this invariant the reader may consult \cite{Raghavan}} (Raghavan; \cite{Raghavan})
        \end{itemize}
        also
        \begin{itemize}
            \item[--] $\min\{\bbb,\cov(\Null)\}\leq\cov^{*}(\Z)$ (Hrušák and Hern\'{a}ndez-Hern\'{a}ndez; \cite{Hernand})
        \end{itemize}
\end{theorem}
Despite all the effort, the following problem remains open.
\begin{question}
    Is $\cov^{*}(\Z)\leq\bbb$ provable in ZFC?
\end{question}
Here, as $\Z\leq_{K}\trN\leq_{K}\SPL$ we get a new upper bound of the covering number of the trace of null ideal. This provides more conviction to the problem stated above.
\begin{corollary}
    $\cov^{*}(\Z)\leq\cov^{*}(\trN)\leq\rrr$
\end{corollary}
Next, we focus on the cofinality of the splitting ideal
\begin{proposition}\label{cofSPL}
    $\cof^{*}(\SPL)=\continuum$
\end{proposition}
\begin{proof}
    Let $\mathcal{A}\subseteq[\omega]^{\omega}$ be an almost disjoint family of cardinality $\continuum$. Suppose that $\kappa=\cof^{*}(\SPL)<\continuum$ and that $\{B_{\alpha}:\alpha<\kappa\}\subseteq[\omega]^{\omega}$ is such that finite unions of sets from $\{S(B_{\alpha}):\alpha<\kappa\}$ is a witness for $\cof^{*}(\SPL)$. Without loss of generality we may assume that each $B_{\alpha}$ is an almost subset of some $A_{\alpha}\in\mathcal{A}$. Let then $A\in\mathcal{A}\setminus\{A_{\alpha}:\alpha<\kappa\}$. It is not difficult to see that no $S(B_{\alpha})$ can cover $S(A)$.
\end{proof}
We will also estimate the $\cov^{+}$ number of the splitting ideal. See \ref{defOfCovPlus} for the definition of $\cov^{+}$ number.
    \begin{proposition}
    $\cov(\Meager)\leq\cov^{+}(\SPL)\leq\rrr$
    \end{proposition}
    \begin{proof}
        The inequality $\cov^{+}(\SPL)\leq\rrr$ follows from $\cov^{+}(\SPL)\leq\cov^{*}(\SPL)=\rrr$.\\
        To prove the inequality $\cov(\Meager)\leq\cov^{+}(\SPL)$ assume that $\kappa<\cov(\Meager)$ and that we are given a family $\{A_{\alpha}:\alpha<\kappa\}\subseteq[\omega]^{<\omega}$. Let $\mathbb{C}$ be the forcing notion for adding perfect tree of Cohen reals i.e. $\mathbb{C}$ consists of finite subtrees of $\cantor$ ordered by end-extension. Note that $\mathbb{C}$ is forcing equivalent to Cohen forcing. Notice also, that for every $\alpha<\kappa$ the set
        \begin{center}
            $D_{\alpha}=\{F\in\mathbb{C}:$ no terminal of $F$ is a member of $S(A_{\alpha})\}$
        \end{center}
        is dense in $\mathbb{C}$. As $\kappa<\cov(\Meager)$, there is a filter $G\subseteq\mathbb{C}$ intersecting all $D_{\alpha}$'s, $\alpha<\kappa$. Define the tree $T\subseteq\cantor$ as $\bigcup G$. It is not difficult to see that $T$ is $\SPL$-positive and the intersection $T\cap S(A_{\alpha})$ is finite for every $\alpha<\kappa$.
    \end{proof}
    I do not know if the $+$-covering of Splitting ideal is simply equal to $\rrr$.
    \begin{question}
        Is $\cov^{+}(\SPL)=\rrr$ ? 
    \end{question}
    We will be also interested in the $\omega$-versions of the standard cardinal invariants of ideals. For an ideal $\J$ we define:\\
    $\add_{\omega}^*(\J) = \min\{ |\mathcal{F}| : \mathcal{F} \subseteq \J \text{ and } \forall \{X_{n}:n\in\omega\} \subseteq \J$ $\exists A \in \mathcal{F}$ $\forall n\in\omega$ $(A \nsubseteq^* X_{n}) \}$,\\
    $\non_{\omega}^*(\J) = \min\{ |\mathcal{F}| : \mathcal{F} \subseteq [\omega]^{\omega} \text{ and } \forall \overline{A}\in[\J]^{\omega}$ $ \exists F\in\mathcal{F}$ $\forall A\in\overline{A}$ $(A\cap F=^*\emptyset) \}$,\\
    $\cof_{\omega}^*(\J) = \min\{ |\mathcal{F}| : \mathcal{F} \subseteq [\J]^{\omega} \text{ and } \forall \overline{A} \in [\J]^{\omega}$ $ \exists \overline{F}\in\mathcal{F}$ $\forall A\in\overline{A}$ $\exists F\in\overline{F}$ $(A\subseteq^{*}F)\}$.\\
    
    These invariants has several applications to $\sigma$-ideals of reals (see for example \cite{BS99}, \cite{CIESGAPPOYAMAZOEMARTINEZ} or \cite{CIESMARTINEZ}) or ideal convergence (see \cite{FilipowKwela}). Clearly, the relations between the $\omega$-versions are the same as for the original ones and additionally we have that $\mathrm{inv}^{*}(\J)\leq\mathrm{inv}_{\omega}^{*}(\J)$ for every $\mathrm{inv}\in\{\add,\non,\cof\}$. It is not difficult to see that for $P$-ideals these $\omega$-invariants are equal to the original ones (see \cite{CIESMARTINEZ}).\\
    Regarding the $\omega$-invariants of $\SPL$ we have the following
    \begin{proposition}
        $\add_{\omega}^{*}(\SPL)=\omega_{1}$ and $\cof_{\omega}^{*}(\SPL)=\continuum$.
    \end{proposition}
    \begin{proof}
        To show that $\add_{\omega}^{*}(\SPL)\leq\omega_{1}$ let $\{A_{\alpha}:\alpha<\omega_{1}\}$ be an almost disjoint family on $\omega$. We claim that the collection $\{S(A_{\alpha}):\alpha<\omega_{1}\}$ witnesses $\add_{\omega}^{*}(\SPL)$. Suppose then that we are given countably many elements of $\SPL$. Without loss of generality, we may assume that they are of form $S(B_{n})$, where $B_{n}\in[\omega]^{\omega}$ for each $n$. Note that $A\subseteq B$ implies that $S(B)\subseteq S(A)$ and therefore we may shrink $B_{n}$'s if necessary and assume that each $B_{n}$ is a subset of some $A_{\alpha_{n}}$. Pick $\alpha\in\omega_{1}\setminus\{\alpha_{n}:n\in\omega\}$. For each $n\in\omega$, as $A_{\alpha}$ and $B_{n}$ are almost disjoint, we cant have $S(A_{\alpha})\subseteq^{*}S(B_{n})$.\\
        The equality $\cof_{\omega}^{*}(\SPL)=\continuum$ simply follows from $\continuum=\cof^{*}(\SPL)\leq\cof^{*}_{\omega}(\SPL)$ (by Proposition \ref{cofSPL}).
    \end{proof}
    We will now calculate the $\non^{*}_{\omega}$ of the splitting ideal. To do this we will need the following version of the splitting number. Call a family $\mathcal{A}\subseteq[\omega]^{\omega}$ $\omega$-splitting if for every $\{A_{n}:n\in\omega\}\subseteq[\omega]^{\omega}$ there is $A\in\mathcal{A}$ such that $A$ splits every $A_{n}$, $n\in\omega$. The smallest size of $\omega$-splitting family is denoted by $\sss_{\omega}$. Clearly $\sss\leq\sss_{\omega}$ but it is a long-standing open problem whether the two are equal.
    It turns out that the $\non^{*}_{\omega}$ of the Splitting ideal is exactly equal to this invariant.
    \begin{proposition}
        $\non^{*}_{\omega}(\SPL)=\sss_{\omega}$
    \end{proposition}
    \begin{proof}
        To show that $\non^{*}_{\omega}(\SPL)\leq\sss_{\omega}$ assume that $\mathcal{A}=\{A_{\alpha}:\alpha<\kappa\}\subseteq[\omega]^{\omega}$ is $\omega$-splitting. Let $X_{\alpha}=\{1_{A_{\alpha}}|_{n}:n\in\omega\}\in[2^{<\omega}]^{\omega}$. We claim that $\{X_{\alpha}:\alpha<\kappa\}$ is a wittness for $\non^{*}_{\omega}(\SPL)$. To see this, assume that $\{B_{n}:n\in\omega\}\subseteq[\omega]^{\omega}$ is given. As $\mathcal{A}$ is $\omega$-splitting, there is $\alpha<\kappa$ such that $A_{\alpha}$ splits all $B_{n}$'s. But then $X_{\alpha}\cap S(B_{n})$ is finite, because if $k\in\omega$ is such that $k\cap(A_{\alpha}\cap B_{n})\neq\emptyset\neq k\cap(B_{n}\setminus A_{\alpha})$, then we have that $(X_{\alpha}\setminus 2^{\leq k})\cap S(B_{n})=\emptyset$.\\
        To show that $\non^{*}_{\omega}(\SPL)\geq\sss_{\omega}$ assume that $\kappa<\sss_{\omega}$ and a family $\{A_{\alpha}:\alpha<\kappa\}\subseteq\SPL$ is given. For every $\alpha<\kappa$ we will construct two reals $x_{\alpha}$ and $y_{\alpha}$ in $\Cantor$ as follows. By shrinking each $A_{\alpha}$ is necessary, we may assume that $A_{\alpha}$ either forms a chain or an antichain in $\cantor$. If $A_{\alpha}$ is a chain, let $x_{\alpha}=y_{\alpha}$ be the union of this chain. If $A_{\alpha}$ is an antichain, say $A_{\alpha}=\{\sigma^{\alpha}_{n}:n\in\omega\}$, then by compactness argument we may assume that $A_{\alpha}$ converges, in a sense that there is a real $x_{\alpha}\in\Cantor$ such that for every $N\in\omega$ almost all elements of $A_{\alpha}$ extends $x_{\alpha}|_{N}$. Define then $(a^{\alpha}_{n})_{n\in\omega}$ where $a^{\alpha}_{n}$ is the smallest $l\in\omega$ such that $\sigma^{\alpha}_{n}(l)\neq x_{\alpha}(l)$. Again, by shrinking $A_{\alpha}$ if necessary, we may assume that $|\sigma^{\alpha}_{n}|<a^{\alpha}_{n+1}<|\sigma^{\alpha}_{n+1}|$ holds for every $n\in\omega$. Let then $y_{\alpha}\in\Cantor$ be such that
        if $l\in[|\sigma^{\alpha}_{n}|,a^{\alpha}_{n+1})$ for $n\in\omega$, then 
        $y_{\alpha}(l)=x_{\alpha}(l)$, and if $l\in[a^{\alpha}_{n+1},|\sigma^{\alpha}_{n+1}|)$ for $n\in\omega$, then $y_{\alpha}(l)=\sigma^{\alpha}_{n+1}(l)$.\\
        As $\kappa<\sss_{\omega}$ there is $\{B_{n}:n\in\omega\}\subseteq[\omega]^{\omega}$ with the property that for every $\alpha<\kappa$ there is $n_{\alpha}\in\omega$ such that both $x_{\alpha}$ and $y_{\alpha}$ are constant on $B_{n_{\alpha}}$ modulo finite set. Applying $\kappa<\sss_{\omega}$ twice inside of every $B_{n}$ we find a collection $\{C^{n}_{m}:m\in\omega\}$ such that for every $m\in\omega$ we have:
        \begin{itemize}
            \item[--] $C^{n}_{m}\subseteq B_{n}$ is infinite,
            \item[--] $C^{n}_{m}\subseteq^{*}\bigcup_{l}I_{2l+1}$ or $C^{n}_{m}\subseteq^{*}\bigcup_{l}I_{2(2l)}$ or $C^{n}_{m}\subseteq^{*}\bigcup_{l}I_{2(2l+1)}$
        \end{itemize}
        We claim that $\{S(C^{n}_{m}\setminus k):m,n,k\in\omega\}$ has the property that for every $\alpha<\kappa$ there are $m,n,k\in\omega$ such that $A_{\alpha}\cap S(C^{n}_{m}\setminus k)$ is infinite. To see this assume that $\alpha<\kappa$ is given and let $k\in\omega$ be such that $x_{\alpha}$ and $y_{\alpha}$ are constant on $B_{n_{\alpha}}\setminus k$ and that $C^{n_{\alpha}}_{m}\setminus k$ is contained in one of the sets $\bigcup_{l}I_{2l+1}$, $\bigcup_{l}I_{2(2l)}$ or $\bigcup_{l}I_{2(2l+1)}$. To see that $|A_{\alpha}\cap S(C^{n_{\alpha}}_{m}\setminus k)|=\omega$ notice by the argument similar to the one at the end of Proposition \ref{covSplIsRRR}, we have that: if $C^{n_{\alpha}}_{m}\setminus k\subseteq \bigcup_{l}I_{2l+1}$ then $A_{\alpha}\subseteq S(C^{n_{\alpha}}_{m}\setminus k)$; if $C^{n_{\alpha}}_{m}\setminus k\subseteq \bigcup_{l}I_{2(2l)}$ then $\{\sigma^{\alpha}_{2l+1}:l\in\omega\}\subseteq S(C^{n_{\alpha}}_{m}\setminus k)$; and if $C^{n_{\alpha}}_{m}\setminus k\subseteq \bigcup_{l}I_{2(2l+1)}$ then $\{\sigma^{\alpha}_{2l+1}:l\in\omega\}\subseteq S(C^{n_{\alpha}}_{m}\setminus k)$.
    \end{proof}
    The dual of this inequality does not make sense because $\non^{*}(\SPL)=\omega$. However, the same idea can be applied to get the following lower bound for $\non^{*}(tr(\Null))$.
\begin{proposition}
    $\sss\leq \non^{*}(\trN)\leq\non^{*}(\Z)$
\end{proposition}
\begin{proof}
    By the previous Proposition and the fact that for $P$-ideals $\non^{*}$ and $\non^{*}_{\omega}$ are equal, we get that $\sss\leq\sss_{\omega}=\non^{*}_{\omega}(\SPL)\leq\non^{*}_{\omega}(\trN)=\non^{*}(\trN)$.
\end{proof}

\section{Antichain numbers of $\mathcal{P}(\omega)/\J$}
Many researchers were interested in cardinal invariants of Van Douwen's diagram for algebra $\mathcal{P}(\omega)/ \J$ (see for example \cite{HrusBalcar}, \cite{Stevo}, \cite{FilipowNice}, \cite{BrendleDiego}). In particular, there is an increasing interest in the antichain numbers of such algerbas.
\begin{definition}
    Suppose that $\J$ is and ideal on $\omega$. We define $\aaa(\J)$, the antichain number of $\J$, to be the smallest size of an uncountable family of $\J$-positive sets, maximal with respect to the property that the intersection of any two of its elements is in $\J$.
\end{definition}
Let us gather known lower boundaries for the $\J$-antichain numbers.
\begin{theorem}(Farkas, Sokoup; \cite{FarkSoukup})\label{FarkasSokoup}
    $\bbb\leq\aaa(\J)$ for all analytic P-ideals.
\end{theorem}
Regarding the Fubini product of finite sets ideal we have
\begin{theorem}(Steprāns; \cite{Steprans})\label{BrendleFFantichains}
    $\bbb\leq\aaa(\FF)$
\end{theorem}
These two examples suggest, that the classical result saying that $\bbb\leq\aaa$ extends quite well from the algebra $\mathcal{P}(\omega)/fin$ to the algebras $\mathcal{P}(\omega)/\J$. However, this is not completely true as $\bbb\leq\aaa(\J)$ may be not provable even for $F_{\sigma}$-ideals. This is due to the following, unpublished result of Brendle.
\begin{theorem}(Brendle; \cite{BrendleAntichains})
    $\aaa(\ED_{fin})<\bbb$ holds in Hechler's model.
\end{theorem}
Another example of an ideal for which $\bbb\leq\aaa(\J)$ is not provable is the nowhere-dense ideal, due to the following result of Steprāns.
\begin{theorem}(Steprāns; \cite{Steprans})
    The following holds:
    \begin{itemize}
        \item[--] $\ppp\leq\aaa(\nwd)$ holds in ZFC,
        \item[--] $\aaa(\nwd)<\bbb$ holds in Laver's model.
    \end{itemize}
\end{theorem}
For ideals generated by m.a.d. families we have the following.
\begin{theorem}(Brendle, Castro, Hrušák, Mejía; \cite{BrendleCastro})
    For any m.a.d. family $\mathcal{A}$, the inequality $\bbb\leq\aaa(\mathcal{I}(\mathcal{A}))$ holds.
\end{theorem}
We will show another lower bound (which implies some of the results stated above) for a wide class of ideals (including almost all ideals that we are interested in this article). We will need the following cardinal invariant.
\begin{definition}\label{defOfCovPlus}
For an ideal $\J$ define:
\begin{itemize}
        \item[--] $\cov^{+}(\J)=\min\{|\mathcal{F}|:\mathcal{F}\subseteq\J \forall X\in\J^{+}$ $\exists F\in\mathcal{F}$ such that $|X\cap F|=\omega\}$,
        \item[--] $\cov^{+}_{h}(\J)=\min\{\cov^{+}(\J|_{X}):X\in\J^{+}\}$
\end{itemize}
\end{definition}
The relation between the two invariant is closely related to the notion of \emph{K-uniformity}. An ideal $\J$ is \emph{K-uniform} if $\J|_{F}\leq\J$ for every $F\in\J^{+}$. 
The following observation is easy to verify (see \cite{FarkasZdomskyy} for the proof).
\begin{proposition}
    The following hold:
    \begin{itemize}
        \item[--] $\add^{*}(\J)\leq \cov^{+}(\J)\leq\cov^{*}(\J)$
        \item[--] $\cov^{+}_{h}(\J)\leq \cov^{+}(\J)$,
        \item[--] $\J_{0}\leq_{K}\J_{1}$ implies that $\cov^{+}(\J_{1})\leq\cov^{+}(\J_{0})$,
        \item[--] if $\J$ is K-uniform, then $\cov^{+}_{h}(\J)=\cov^{+}(\J)$
    \end{itemize}
\end{proposition}
In \cite{FarkasZdomskyy} Farkas and Zdomskyy systematically investigated cardinal invariants $\cov^{+}$ for various Borel ideal and related forcing properties. We will gather here the values of some of this numbers as we will need them later on.
\begin{theorem}(Farkas and Zdomskyy, \cite{FarkasZdomskyy})\label{FarkasZdomskyyOnCovPlus}
    We have the following:
    \begin{itemize}
        \item[--] $\cov^{+}(\FF)=\omega$,
        \item[--] $\cov^{+}(\nwd)=\add(\Meager)$,
        \item[--] $\cov^{+}(\mathcal{S})=\non(\Null)$,
        \item[--] $\cov^{+}(\ED)=\non(\Meager)$,
        \item[--] $\cov^{+}(\ED_{fin})=\cov^{*}(\ED_{fin})$,
        \item[--] $\cov^{+}(\conv)=\continuum$,
        \item[--] $\cov^{+}(\mathcal{R})=\continuum$
    \end{itemize}
    Also, $\ppp\leq\cov^{+}(\J)$ for any ideal $\J$ with uncountable $\cov^{+}(\J)$.
\end{theorem}
Regarding $\cov^{+}$ numbers for our analytic $P$-ideals we have
\begin{theorem}(Cieślak, Farkas, Zdomskyy, \cite{Killing_Ideals_Softly})
    We have the following:
    \begin{itemize}
        \item[--] $\cov^{+}(\SUM)\leq \non(\E)$,
        \item[--] $\cov^{+}(\trN)\leq\cov(\E)$,
        \item[--] $\cov^{+}(\Z)\leq\cov(\Meager)$
    \end{itemize}
    also $\add(\Null)\leq\cov^{+}(\Z)\leq\cov^{+}(\trN)\leq \cov^{+}(\SUM)$.
\end{theorem}
For more information of $\cov^{+}$ numbers as well as different consistency result regarding these, the reader may consult \cite{FarkasZdomskyy} or \cite{Killing_Ideals_Softly}. As a corollary of the above, we get
\begin{proposition}
    We have the following:
    \begin{itemize}
        \item[--] $\cov^{+}_{h}(\FF)=\omega$,
        \item[--] $\cov^{+}_{h}(\nwd)=\add(\Meager)$,
        \item[--] $\cov^{+}_{h}(\ED_{fin})=\cov^{+}(\ED_{fin})$,
        \item[--] $\cov^{+}_{h}(\ED)=\cov^{+}(\ED_{fin})$,
        \item[--] $\cov^{+}_{h}(\conv)=\omega$,
        \item[--] $\cov^{+}_{h}(\BI)=\omega$,
        \item[--] $\cov^{+}_{h}(\Z)=\cov^{+}(\Z)$,
        \item[--] $\cov^{+}_{h}(\trN)=\cov^{+}(\trN)$
    \end{itemize}
\end{proposition}
\begin{proof}
    The equality $\cov^{+}_{h}(\FF)=\omega$ follows from Theorem \ref{FarkasZdomskyyOnCovPlus}. The equalities $\cov^{+}_{h}(\nwd)=\add(\Meager)$ and $\cov^{+}_{h}(\ED_{fin})=\cov^{+}(\ED_{fin})$ follow from Theorem \ref{FarkasZdomskyyOnCovPlus} and homogeneity of the ideals $\nwd$ and $\ED_{fin}$. The equality $\cov^{+}_{h}(\ED)=\cov^{+}(\ED_{fin})$ follows from the fact that any $\ED$-positive set contains $\ED$-positive set $F$ such that $\ED|_{F}=\ED_{fin}$. The equalities $\cov^{+}_{h}(\conv)=\omega$ and $\cov^{+}_{h}(\BI)=\omega$ follows from the fact that for these two ideals $\J$ any positive set contains a positive set $F$ such that $\J|_{F}=\FF$ and $\cov^{+}(\FF)=\omega$ by Theorem \ref{FarkasZdomskyyOnCovPlus}. The last two equalities follow from the fact that $\Z$ and $\trN$ are K-uniform (in \cite{Meza} see Proposition 2.1.11 and Theorem 2.1.17).
\end{proof}
In order to prove our main lower bound for numbers $\aaa(\J)$ we will need to introduce the following class of ideals.
\begin{definition}
    Suppose that $\J$ is and ideal on $\omega$. We will say that $\J$ is \emph{good} if for every uncountable family $\mathcal{F}\subseteq\J^+$ has a subfamily $\{B_{n}:n\in\omega\}\subseteq\mathcal{F}$ such that for every $f\in\Baire$ there is a sequence $\{C_{n}:n\in\omega\}\subseteq\J$, $C_{n}\subseteq B_{n}\setminus f(n)$ such that $C:=\bigcup_{n}C_{n}$ is $\J$-positive.
\end{definition}
We are now ready to prove the main result of this section.
\begin{theorem}\label{MAINthmAntichains}
    If an ideal $\J$ is good, then $\min\{\cov^{+}_{h}(\J),\bbb\}\leq\aaa(\J)$.
\end{theorem}
\begin{proof}
    Let $\kappa<\min\{\cov^{+}_{h}(\J),\bbb\}$ and let $\{A_{\alpha}:\alpha<\kappa\}$ be an antichain in $\J^+$. As $\kappa<\cov^{+}_{h}(\J)$ there are $\{B_{\alpha}:\alpha<\kappa\}\subseteq\J^+$ such that $B_{\alpha}\subseteq A_{\alpha}$ and $B_{\alpha}\cap A_{\beta}=^*\emptyset$ for $\alpha\neq\beta$. Suppose that $\{B_{n}:n<\omega\}$ is as in definition of goodness. Without loss of generality we may assume that the intersection $B_{n}\cap\bigcup_{i\leq n}A_{i}$ is empty for each $n\in\omega$. For each $\alpha\in\kappa\setminus\omega$ define $f_{\alpha}\in\Baire$ such that $B_{n}\cap A_{\alpha}\subseteq f_{\alpha}(n)$ for each $n\in\omega$. Let $f\in\Baire$ be a function which dominates all of $f_{\alpha}$'s. By the property of being good there is a sequence $\{C_{n}:n\in\omega\}\subseteq\J$, $C_{n}\subseteq B_{n}\setminus f(n)$ such that $C:=\bigcup_{n}C_{n}$ is $\J$-positive. Then, clearly the intersection of $C$ with any $A_{\alpha}$, $\alpha<\kappa$ is in the ideal $\J$ which finishes the proof.
\end{proof}
As a corollary we get the following lower bound of $\aaa(\J)$'s for good ideals.
\begin{corollary}
    $\ppp\leq\aaa(\J)$ for good ideal $\J$ with uncountable $\cov^{+}_{h}(\J)$.
\end{corollary}
\begin{proof}
    Let $A\in\J^{+}$. If $\cov^{+}(\J|_{A})$ is uncountable, then by the Corollary 5.2 of \cite{FarkasZdomskyy}, the Mathias forcing with the dual filter $\mathbb{M}((\J|_{A})^{*})$ increases $\cov^{+}(\J|_{A})$. As $\mathbb{M}((\J|_{A})^{*})$ is $\sigma$-centered forcing notion, we have that $\ppp\leq\cov^{+}(\J|_{A})$. As also $\ppp\leq\bbb$ we get that $\ppp\leq\min\{\bbb,\cov^{+}_{h}(\J)\}\leq\aaa(\J)$.
\end{proof}
We will show that (almost) all of the ideals considered in the article are good. The only example of an ideal that I do not know if it is good is the 'splitting ideal'.
\begin{question}
    Find an example of a definable ideal $\J$ which is not good. Is the ideal $\SPL$ any good?
\end{question}
We will start by showing the goodness of $F_{\sigma}$ and analytic $P$-ideals.
\begin{proposition}
    All $F_{\sigma}$-ideals are good.
\end{proposition}
\begin{proof}
    By a result of Mazur (Theorem \ref{Mazur}) all $F_{\sigma}$-ideals are of the form $Fin(\phi)=\{X\subseteq\omega:\phi(X)<\infty\}$ for some l.s.c.s.m. $\phi$. Assume that $\mathcal{F}$ is a collection of $Fin(\phi)$-positive sets. Let $\{B_{n}:n\in\omega\}\subseteq\mathcal{F}$ be arbitrary.
    For every $f\in\omega^{\omega}$ it is easy to pick finite sets $C_{n}$'s such that $C_{n}\subseteq B_{n}\setminus f(n)$ and $\phi(C_{n})>n$. Clearly then $\phi(C)=\infty$.
\end{proof}
It follows that the ideals $\ED$, $\ED_{fin}$, $\SUM$, $\RandomGraph$, $\Solecki$ and $\WANDER$ are good.
\begin{proposition}
    All analytic $P$-ideals are good.
\end{proposition}
\begin{proof}
    Recall a result of Solecki saying that an analytic $P$-ideal $\J$ is of form $Exh(\phi)=\{X\subseteq\omega:lim_{k}\phi(X\setminus k)=0\}$ for some l.s.c.s.m. $\phi$. Assume that $\mathcal{F}\subseteq Exh(\phi)^{+}$. Let $\{B_{n}:n\in\omega\}\subseteq\mathcal{F}$ be such that for some $\epsilon>0$ we have $lim_{k}\phi(B_{n}\setminus k)>\epsilon$ for each $n\in\omega$. For each $f\in\omega^{\omega}$ and $n\in\omega$ it is not difficult to pick $C_{n}\subseteq B_{n}\setminus f(n)$ with $\phi(C_{n})>\epsilon/2$. Clearly then $lim_{k}\phi(C\setminus k)=0$.
\end{proof}
It follows that the ideals $\Z$ and $\trN$ are good. It is worth mentioning that in the realm of analytic P-ideals, the estimate given by Theorem \ref{MAINthmAntichains} is weaker than the one given by Farkas and Sokoup in Theorem \ref{FarkasSokoup}. 
\begin{proposition}
    The ideal $\nwd$ is good.
\end{proposition}
\begin{proof}
    First assume that $\mathcal{F}$ is a collection of $\nwd$-positive sets. Let $\{B_{n}:n\in\omega\}\subseteq\mathcal{F}$ be such that all $B_{n}$'s are dense below the same $\sigma^{*}$. Clearly for every $f\in\omega^{\omega}$ it is easy to pick finite $C_{n}$'s such that $C_{n}\subseteq B_{n}\setminus f(n)$ and for every $\sigma\in \{\tau\in2^{n}:\sigma^{*}\subseteq\tau\}$ there is $\rho\in C_{n}$ which extends $\sigma$.
\end{proof}
In \cite{Steprans} Steprāns proved that $\ppp\leq \aaa(\nwd)$. On the other hand by a result of Balcar, Hernández-Hernández and Hrušák $\cov^{*}_{+}(\nwd)=\add(\Meager)$, so we obtain the following improvement of the result of Steprāns.
\begin{corollary}
    $\add(\Meager)\leq \aaa(\nwd)$
\end{corollary}
This result, by using a different method, has been independently obtained by Jonathan Cancino Manriquez during the work on this paper. Next, we move to the ideals $\FF$, $\BI$ and $\conv$.
\begin{proposition}
    The ideals $\FF$ and $\BI$ are good.
\end{proposition}
\begin{proof}
    For the $\FF$ case, assume that $\mathcal{F}$ is an uncountable collection of $\FF$-positive sets. Let $\{B_{n}:n\in\omega\}$ be any countable sub-collection of $\mathcal{F}$. Given $f\in\omega^{\omega}$ inductively build an increasing sequence $\{c_{n}:n\in\omega\}$ of natural numbers such that $C_{n}:=B_{n}\cap(\{c_{n}\}\times\omega)\setminus f(n)$ is infinite.\\
    For the $\BI$ case, assume that $\{A_{\alpha}:\alpha<\omega_{1}\}$ is a collection of $\BI$-positive sets. Consider the following two cases. If the set
    \begin{center}
        $W=\{\alpha<\kappa: \exists i_{\alpha}\in\omega$ such that $\{(j,k):(i_{\alpha},j,k)\in A_{\alpha}\}\in\FF^{+}\}$
    \end{center}
    is uncountable, pick $\{\alpha_{n}:n\in\omega\}\subseteq W$ such that all $i_{\alpha_{n}}$'s are equal to the same $i^{*}\in\omega$. Let $f\in\Baire$. Inductively choose an increasing sequence $\{j_{n}:n\in\omega\}$ such that for every $n\in\omega$ we have that $(\{i^{*}\}\times\{j_{n}\}\times\omega)\cap A_{\alpha_{n}}$ is infinite. Then, let $\{C_{n}:n\in\omega\}$ be such that $C_{n}=(\{i^{*}\}\times\{j_{n}\}\times\omega)\cap A_{\alpha_{n}}\setminus f(n)$
    Then the sequence $\{C_{n}:n\in\omega\}$ is as required.\\
    If the set $W$ is countable, let $\{\alpha_{n}:n\in\omega\}\subseteq \omega_{1}\setminus W$ and choose an increasing sequence $\{i_{n}:n\in\omega\}$ such that $A_{\alpha_{n}}\cap(\{i_{n}\}\times\omega^{2})$ is infinite for every $n\in\omega$. For $f\in\Baire$ let $\{C_{n}:n\in\omega\}$ be such that $C_{n}=A_{\alpha_{n}}\cap(\{i_{n}\}\times\omega^{2})\setminus f(n)$. Again, $\{C_{n}:n\in\omega\}$ is as required.
\end{proof}
Even though both ideals $\FF$ and $\BI$ are good, the Theorem \ref{MAINthmAntichains} does not give interesting lower boundaries for $\aaa(\FF)$ and $\aaa(\BI)$ as both numbers $\cov^{+}_{h}(\FF)$ and $\cov^{+}_{h}(\BI)$ are equal to $\omega$. However, by the Theorem \ref{BrendleFFantichains} the inequality $\bbb\leq\aaa(\FF)$ holds. This motivates the question whether $\bbb\leq\aaa(\BI)$ is true as well? We will show this directly in the following
\begin{theorem}\label{bbbLeqAAABi}
    $\bbb\leq\aaa(\BI)$
\end{theorem}
\begin{proof}
    Assume that $\kappa<\bbb$ and that a $\BI$-antichain $\{A_{\alpha}:\alpha<\kappa\}\subseteq\BI^{+}$ is given.\\
    We will say that $\alpha<\kappa$ is of type 1, if there is an $i_{\alpha}\in\omega$ such that the set $\{(j,k):(i_{\alpha},j,k)\in A_{\alpha}\}$ is $\FF$-positive. In such case the set $\{(j,k):(i_{\alpha},j,k)\in A_{\alpha}\}$ contains a set of the form $F_{\alpha}=\bigcup_{n}\{j^{\alpha}_{n}\}\times F^{\alpha}_{n}$ where $\{j^{\alpha}_{n}:n\in\omega\}\in[\omega]^{\omega}$ and $F^{\alpha}_{n}\in[\omega]^{\omega}$ for every $n\in\omega$.\\
    We will say that $\alpha<\kappa$ is of type 2, if it is not of type 1. Note that if $\alpha$ is of type 2, then the set $I_{\alpha}=\{i\in\omega:|(\{i\}\times\omega^{2})\cap A_{\alpha}|=\omega\}$ is infinite.\\
    
    Consider now the following two cases:\\
    CASE 1: the set $W=\{\alpha<\kappa: \alpha$ is of type 1$\}$ is uncountable.\\
    Without loss of generality we may assume that $\omega\subseteq W$ and $i_{\alpha}$'s for $\alpha\in W$ are equal to the same $i^{*}\in\omega$. For every $\alpha<\kappa$ and $l\in\omega$ with $l\neq\alpha$ there are $K^{\alpha}(l)\in\omega$ and $f^{\alpha}_{l}\in\Baire$ such that:
    \begin{center}
    $\{(j,k)\in\omega^{2}:(i^{*},j,k)\in A_{l}\cap A_{\alpha}\}\subseteq\{(j,k)\in\omega^{2}:j<K^{\alpha}(l)$ or $k\leq f^{\alpha}_{l}(j)\}$
    \end{center}
    As $\kappa<\bbb$ there is $f\in\Baire$ that dominates all $f^{\alpha}_{l}$'s. Then, for every $\alpha\in\kappa\setminus\omega$ we define $g_{\alpha}\in\Baire$ in a way that for all $j\geq g_{\alpha}(l)$ we have $f^{\alpha}_{l}(j)<f(j)$. Again, as $\kappa<\bbb$, we can find $g\in\Baire$ that dominates all functions $g_{\alpha}$ and $K^{\alpha}$.\\
    Inductively we construct sequences $\{j_{l}:l\in\omega\}$ and $\{C_{l}:l\in\omega\}$ such that for every $l\in\omega$ we have:
    \begin{itemize}
        \item[--] $j_{l+1}>j_{l}>g(l)$,
        \item[--] $C_{l}=F_{l}\cap (\{j_{l}\}\times(\omega\setminus f(j_{l}))$ is infinite
    \end{itemize}
    When the induction is finished define $C=\bigcup_{l}\{j_{l}\}\times C_{l}$ and $A=\{i^{*}\}\times C$. Note that the set $A$ is $\BI$-positive because $C$ is $\FF$-positive. We claim that for every $\alpha<\kappa$ we have $A\cap A_{\alpha}\in\BI$ which will finish the first case. Fix $\alpha<\kappa$. If $\{(j,k)\in\omega^{2}:(i^{*},j,k)\in A_{\alpha}\}\in\FF$, there is nothing to do. Assume then, that $\{(j,k)\in\omega^{2}:(i^{*},j,k)\in A_{\alpha}\}$ is $\FF$-positive. Let $L\in\omega$ be large enough so that for all $l>L$ we have $g(l)\geq\max\{g_{\alpha}(l),K^{\alpha}(l)\}$. But then
    \begin{center}
        $\{(j,k)\in\omega^{2}:(i^{*},j,k)\in A\cap A_{\alpha}\}\subseteq\bigcup_{l<L}C_{l}$
    \end{center}
    This is because if $l>L$, then $j_{l}>\max\{g_{\alpha}(l),K^{\alpha}(l)\}$ and we have that
    \begin{center}
    $C_{l}\cap\{(j,k)\in\omega^{2}:(i^{*},j,k)\in A_{\alpha}\}\subseteq C_{l}\cap \{(j,k):j<K^{\alpha}(l)$ or $k<f^{\alpha}_{l}(j)\}$
    \end{center}
    Notice that the latter is empty because $C_{l}\subseteq\{j_{l}\}\times(\omega\setminus f(j_{l}))$ and $f(j_{l})>f^{\alpha}_{l}(j_{l})$ for all $j_{l}>K^{\alpha}(l)$. As $\bigcup_{l<L}C_{l}\in\FF$, the first case is finished.\\
    
    CASE 2: $W$ is countable.\\
    Let $\{A'_{n}:n\in\omega\}$ enumerate all of the sets $\{A_{\alpha}:\alpha\in W\}$. Then, by removing the set $\{A'_{n}:n\in\omega\}$ from the original enumeration, we may assume that the intersection $\{A'_{n}:n\in\omega\}\cap\{A_{\alpha}:\alpha<\kappa\}$ is empty.\\
    For every $\alpha\in\kappa\setminus\omega$ define $g_{\alpha}\in\Baire$ such that for every $l\in\omega$ for all $i\geq g_{\alpha}(l)$ the set $\{(j,k):(i,j,k)\in A_{l}\cap A_{\alpha}\}$ is finite. Then, for all $\alpha<\kappa$ and $l\in\omega$ with $\alpha\neq l$ define $f^{\alpha}_{l}\in\Baire$ such that $\{(j,k):(i,j,k)\in A_{l}\cap A_{\alpha}\}\subseteq f^{\alpha}_{l}(i)$ for all $i>g_{\alpha}(l)$.
    As $\kappa<\bbb$ we can find $f\in\Baire$ that dominates all such $f^{\alpha}_{l}$'s.
    For every $\alpha\in\kappa\setminus\omega$ define then $h_{\alpha}\in\Baire$ such that for every $i>h_{\alpha}(l)$ we have $f(i)>f^{\alpha}_{l}(i)$. As $\kappa<\bbb$ there is $g\in\Baire$ that dominates all $g_{\alpha}$'s and $h_{\alpha}$'s.\\
    Inductively we construct an increasing sequence $\{i_{l}:l\in\omega\}\subseteq\omega$ and $\{C_{l}:l\in\omega\}$ such that for every $l\in\omega$ we have:
    \begin{itemize}
        \item[--] $i_{l}\in I_{l}\setminus i_{l-1}$,
        \item[--] $C_{l}=\{(j,k):(i_{l},j,k)\in A_{l}\setminus(\bigcup_{l'\leq l}A'_{l'}\cup f(i_{l}))\}$
        \item[--] $i_{l}>g(l)$
    \end{itemize}
    When induction is over define $A=\bigcup_{l}\{i_{l}\}\times C_{l}$. Note that $A$ is $\BI$-positive as each $C_{l}$ is infinite and $\{i_{l}:l\in\omega\}$ is increasing. By a similar argument as in the first case, we have that $A\cap A_{\alpha}\in\BI$ for every $\alpha\in\kappa$ and $A\cap A'_{n}\in\BI$ for every $n\in\omega$. This finishes the proof.
\end{proof}
We next turn to the convergent ideal.
\begin{proposition}
    $\conv$ ideal is good.
\end{proposition}
\begin{proof}
    Suppose that we are given a family $\{A_{\alpha}:\alpha<\omega_{1}\}$ of almost disjoint $\conv$-positive sets. We will need the following fact asserting that every $\conv$-positive set contains a copy of $\FF$-ideal.
    \begin{lemma}\label{lemmaConvPositive}
        For every $A_{\alpha}$ there is a collection $\{x^{\alpha}_{i,j}:i,j\in\omega\}\subseteq A_{\alpha}$ such that:
        \begin{itemize}
            \item[--] for every $i$ and $j_{0}\neq j_{1}$, $x^{\alpha}_{i,j_{0}}\neq x^{\alpha}_{i,j_{1}}$,
            \item[--] for every $i$, $\{x^{\alpha}_{i,j}:j\in\omega\}$ converges to $y^{\alpha}_{i}$,
            \item[--] for all $i_{0}\neq i_{1}$, $y^{\alpha}_{i_{0}}\neq y^{\alpha}_{i_{1}}$,
            \item[--] $\{y^{\alpha}_{i}:i\in\omega\}$ converges to $z^{\alpha}$
        \end{itemize}
    \end{lemma}
    \begin{proof}[proof of the lemma]
        Notice that $\overline{A}\setminus A$ cannot be finite because then $A\in\conv$. By compactness of $[0,1]$ let $\{y_{i}:i\in\omega\}\subseteq\overline{A}\setminus A$ be an infinite sequence of distinguished points with a limit point $z$. For each $i\in\omega$ choose a sequence $\{x_{i,j}:j\in\omega\}\subseteq A$ of distinguished point converging to $y_{i}$ such that $\{x_{i_{0},j}:j\in\omega\}\cap\{x_{i_{1},j}:j\in\omega\}=\emptyset$ for $i_{0}\neq i_{1}$. This finishes the proof of the claim.
    \end{proof}
    Let $W=\{\alpha_{n}:n\in\omega\}\subseteq\omega_{1}$ be such a countable set that $\{z^{\alpha_{n}}:n\in\omega\}$ converges to a point $w$ and let $f:\omega\rightarrow\mathbb{Q}\cap[0,1]$. Inductively we construct sets $\{C_{n}:n\in\omega\}$ such that for each $n\in\omega$ we have:
    \begin{itemize}
        \item[--] there are $i_{n}\in\omega$ and $N_{n}\in\omega$ such that $C_{n}=\{x^{\alpha_{n}}_{i_{n},j}:j>N_{n}\}$,
        \item[--] $N_{n}$ is large enough so that $C_{n}\cap f[n]=\emptyset$,
        \item[--] $y^{\alpha_{n}}_{i_{n}}\notin\{y^{\alpha_{n}}_{i_{k}}:k<n\}$,
        \item[--] $|y^{\alpha_{n}}_{i_{n}}\cap z|<\frac{1}{n}$.
    \end{itemize}
    The construction is straightforward. Now, as each $C_{n}$ is in $\conv$, it remains to check that $C=\bigcup_{n}C_{n}$ is $\conv$-positive. If $C$ would be covered by finitely many convergent sequences, then the limit points of $C$, it is $\{y^{\alpha_{n}}_{i_{n}}:n\in\omega\}$ and $z$, would be covered by finitely many limit points of the convergent sequences. 
\end{proof}
As in the case with the ideals $\FF$ and $\BI$, the Theorem \ref{MAINthmAntichains} does not give interesting lower boundaries for $\aaa(\conv)$, as invariant $\cov^{+}_{h}(\conv)$ is equal to $\omega$ (this is because for every $A\in\conv^{+}$ there is $B\subseteq A$, $B\in\conv^{+}$ such that $\conv|_{B}$ is Katětov equivalent to $\FF$). However, just as in the case with $\FF$ and $\BI$, it is possible to show that $\aaa(\conv)$ is above $\bbb$ directly.
\begin{theorem}
    $\bbb\leq\aaa(\conv)$
\end{theorem}
\begin{proof}
    Assume that $\kappa<\bbb$ and we are given a $\conv$-antichain $\{A_{\alpha}:\alpha<\kappa\}\subseteq\conv^{+}$. By Lemma \ref{lemmaConvPositive}, for every $\alpha<\kappa$ there is $F_{\alpha}=\{x^{\alpha}_{n,m}:n,m\in\omega\}$ such that
    \begin{itemize}
        \item[--] $F_{\alpha}\subseteq A_{\alpha}$
        \item[--] for every $n\in\omega$, the sequence $\{x^{\alpha}_{n,m}:m\in\omega\}$ converges to some $x^{\alpha}_{n}$,
        \item[--] $\{x^{\alpha}_{n}:n\in\omega\}$ converges to some $x^{\alpha}$.
    \end{itemize}
    For every $\alpha<\kappa\setminus\omega$ define $g_{\alpha}\in\Baire$ such that for every $l\in\omega$ for all $n>g_{\alpha}(l)$ the set $\{x^{l}_{n,m}:m\in\omega\}\cap A_{\alpha}$ is finite. Then for every $\alpha<\kappa\setminus\omega$ and $l\in\omega$ choose a function $f^{\alpha}_{l}\in\Baire$ so that for every $n>g_{\alpha}(l)$ we have that the intersection $\{x^{\alpha}_{n,m}:m>f^{\alpha}_{l}(n)\}\cap A_{\alpha}$ is empty.\\
    As $\kappa<\bbb$, we can find $f,g\in\Baire$ such that $f$ dominates all $f^{\alpha}_{l}$'s and $g$ dominates all $g_{\alpha}$'s. Inductively we will construct $\{n_{l}:l\in\omega\}\subseteq\omega$, $\{z_{l,m}:l,m\in\omega\}\subseteq\mathbb{Q}$ and $\{z_{l}:l\in\omega\}\subseteq[0,1]$ such that for every $l\in\omega$ we have:
    \begin{itemize}
        \item[--] $\{z_{l,m}:m\in\omega\}$ converges to some $z_{l}$
        \item[--] $n_{l}>g(l)$,
        \item[--] $z_{l}\notin\{z_{l'}:l'<l\}$,
        \item[--] $\{z_{l,m}:m\in\omega\}=\{x^{l}_{n_{l},m}:m>f(n_{l})\}$
    \end{itemize}
    When the induction is finished, define $A=\{z_{l,m}:l,m\in\omega\}$. Note that $A$ is $\conv$-positive. By a similar argument as in the Theorem \ref{bbbLeqAAABi}, we have that $A\cap A_{\alpha}\in\conv$ for every $\alpha<\kappa$.
\end{proof}
In the remainder of this section we will show that several other ideal on $\omega$ that naturally appear in the literature are good without being interested in the corresponding antichain numbers. We will first handle some coanalytic ideal.
\begin{definition}
    The ideal of graphs without infinite complete subgraphs is defined as:
    \begin{center}
        $\mathcal{G}_{c}=\{E\subseteq[\omega]^{2}:\forall X\in[\omega]^{\omega}$ $[X]^{2}\nsubseteq E\}$
    \end{center}
\end{definition}
To show that these two ideal are good we will need the following lemma.
\begin{lemma}(Shelah and Spinas \cite{SpHein})\label{SHSP}
    For every $\{A_{\alpha}:\alpha<\omega_{1}\}\subseteq[\omega]^{\omega}$ there is a countable $C\subseteq\omega_{1}$ such that for all $n\in C$ and $ F\in[A_{n}]^{<\omega}$ there are infinitely many $m\in C$ with $F\subseteq A_{m}$.
\end{lemma}
\begin{proposition}
    The ideal $\mathcal{G}_{c}$ is good.
\end{proposition}
\begin{proof}
    Assume that we are given uncountable, almost disjoint family $\mathcal{B}\subseteq\mathcal{G}_{c}^{+}$. Without loss of generality we may assume that for each $B'\in\mathcal{B}$ we have $B'=\{(a,b)\in B^{2}:a<b\}$ for some $B\in[\omega]^{\omega}$. Take arbitrary $\{B'_{n}:n\in\omega\}\subseteq\mathcal{B}$ and $f\in\omega^{\omega}$. We will construct $\{c_{k}:k\in\omega\}\in[\omega]^{\omega}$ and $\{i_{k}:k\in\omega\}\in\Baire$ such that for each $k>0$ we have $c_{k}<c_{k+1}$ and $\{(c_{j},c_{k}):j<k\}\subseteq B_{i_{k-1}}$. Assume that $c_{k}$ and $i_{k-1}$ have been built. By Lemma \ref{SHSP} there are infinitely many $i\in\omega$ such that $\{(c_{j},c_{k}):j<k\}\subseteq B_{i}$. Let $i_{k}$ be one of these. Pick then $c_{k+1}\in\omega$, $c_{k+1}>c_{k}$ such that $\{(c_{j},c_{k}):j<k+1\}\subseteq B_{i_{k}}\setminus f(i_{k})$. When the sequence $\{c_{k}:k\in\omega\}$ is constructed We define sets $C_{n}\subseteq B_{n}\setminus f(n)$ as $C_{n}=\emptyset$ if $n\notin\{i_{k}:k\in\omega\}$ and $C_{n}=\{(c_{j},c_{n+1}):j<n+1\}$ when $k=i_{n}$ for some $n\in\omega$. Clearly, $C=\bigcup_{n}C_{n}\in\mathcal{G}_{c}^{+}$ as $C=\{(c_{j},c_{n}):j<n \land n\in\omega\}$.
\end{proof}
We will next show that more ideals are good.
\begin{definition}
    Branching ideal is the ideal on $\cantor$ generated by branches.
\end{definition}
Note that Branching ideal is not a tall ideal.
\begin{proposition}
    The Branching ideal ideal is good.
\end{proposition}
\begin{proof}
    Assume that $\{A_{\alpha}:\alpha<\omega_{1}\}\subseteq\Br^{+}$ is a collection of almost disjoint sets. By shrinking each $A_{\alpha}$ is necessary we may assume that it forms an antichain in $(\Cantor,\subseteq)$. By Lemma \ref{SHSP} there is $\{\alpha_{n}:n\in\omega\}\subseteq\omega_{1}$ such that
    \begin{center}
        $\forall n\in\omega$ $\forall F\in [A_{\alpha_{n}}]^{<\omega}$ $\exists^{\infty}m$ $F\subseteq A_{\alpha_{n}}$
    \end{center}
    Let $f\in\Baire$ be given. We will inductively build sequences $\{C_{m}:m\in\omega\}$ and $\{N_{n}:n\in\omega\}$ such that for every $m,n\in\omega$ we have:
    \begin{itemize}
        \item[1.] $N_{n}<N_{n+1}$,
        \item[2.] $C_{m}\subseteq A_{\alpha_{m}}\setminus f(m)$,
        \item[3.] $|C_{m}|\leq 1$.
    \end{itemize}
    Suppose that $\{C_{m}:m\leq N_{n}\}$ has been constructed and let $N_{n+1}$ be such that $\bigcup_{m\leq N_{n}}\subseteq A_{\alpha_{N_{n+1}}}$. Such $N_{n+1}$ exists by the property stated in Lemma \ref{SHSP}. Let $C_{N_{n+1}}$ consists of a single element of $A_{\alpha_{N_{n+1}}}\setminus f(N_{n+1})$ that in incompatible with all elements of $\bigcup_{m\leq N_{n}}C_{m}$. Also, let $C_{m}=\emptyset$ for all $N_{n}<m<N_{n+1}$. Then, $\{C_{m}:m\in\omega\}$ is the required sequence as $C=\bigcup_{m}C_{m}$ forms an antichain (which is $\Br^{+}$).
\end{proof}
We will finish this section with a discussion about ideals generated by m.a.d. families.
\begin{proposition}
    Ideals generated by m.a.d. families are good.
\end{proposition}
\begin{proof}
    Let $\{A_{\alpha}:\alpha<\kappa\}$ be uncountable collection of $\mathcal{I}(\mathcal{A})$-positive sets and let $\{B_{n}:n\in\omega\}$ be its arbitrary countable subcollection. Let $\{B'_{n}:n\in\omega\}\subseteq[\omega]^{\omega}$ be disjoint sets such that $B'_{n}\subseteq B_{n}$ and $B'_{n}$ is subset of an unique member of $\mathcal{A}$. Given any $f\in\Baire$ let $C_{n}=B_{n}\setminus f(n)$. Clearly $C_{n}\in\mathcal{I}(\mathcal{A})$ and the set $C=\bigcup_{n}C_{n}$ is $\mathcal{I}(\mathcal{A})$-positive as it infinitely intersects every $B'_{n}\in\mathcal{A}$.
\end{proof}

\begin{proposition}\label{covPlusForMad}
    $\bbb\leq\cov^{+}(\mathcal{I}(\mathcal{A}))\leq\cov^{*}(\mathcal{I}(\mathcal{A}))=|\mathcal{A}|$
\end{proposition}
\begin{proof}
    To show that $\cov^{*}(\mathcal{I}(\mathcal{A}))=|\mathcal{A}|$ notice that $\mathcal{A}$ is a witness for $\cov^{*}(\mathcal{I}(\mathcal{A}))$ and any witness for $\cov^{*}(\mathcal{I}(\mathcal{A}))$ cannot be of size lesser that $|\mathcal{A}|$.\\
    To show that $\bbb\leq\cov^{+}(\mathcal{I}(\mathcal{A}))$ assume that $\mathcal{A}'\subseteq \mathcal{A}$ and $|\mathcal{A}'|<\bbb$. Let $\{A_{n}:n\in\omega\}\subseteq\mathcal{A}\setminus\mathcal{A}'$. For $A\in\mathcal{A}'$ define $f_{A}\in\Baire$ such that for every $n\in\omega$ we have $A\cap A_{n}\subseteq f_{A}(n)$. Dominate all $f_{\alpha}$'s with $f\in\Baire$ and let $B=\bigcup_{n}A_{n}\setminus f(n)$. Clearly $B$ is $\mathcal{I}(\mathcal{A})$-positive as it infinitely intersects all $A_{n}$'s. Also, $B\cap A$ is finite for every $A\in\mathcal{A}'$. This is because if $N\in\omega$ is such that $f_{A}(i)\leq f(i)$ for all $i>N$, then $B\cap A\subseteq\bigcup_{i\leq N}A_{i}\cap f(i)$ which is a finite set.
\end{proof}
As a corollary of the Proposition \ref{covPlusForMad} we obtain the following inequality.
\begin{corollary}(Brendle, Castro, Hrušák, Mejía; \cite{BrendleCastro})
    For any maximal almost disjoint family $\mathcal{A}$, the inequality $\bbb\leq\aaa(\mathcal{I}(\mathcal{A}))$ holds.
\end{corollary}

\section{Antichain numbers of $\mathcal{P}(\omega)/\J$ and Hechler forcing}
In this section, we will show the consistency $\aaa(\J)<\bbb$ for Van der Waerden's ideal and for the linear growth ideal. To do this, we will use finite support iteration of Hechler's forcing.
\begin{definition}
    Hechler forcing $\mathbb{D}$ consists of trees $T\subseteq\baire$ such that for every $s\in T$ extending the stem, the set of successors of $s$ in $T$ is infinite. The order is inclusion.
\end{definition}
It is well-known that Hechler forcing is $\sigma$-centered (thus does not add random reals). It is easy to see that it adds dominating and Cohen reals. Therefore, its iteration of length $\omega_{2}$ over a model of Continuum Hypothesis, produces a model of $\cov(\Null)<\add(\Meager)=\continuum$. In this section, we will show that Hechler's model satisfies $\add(\WANDER)=\omega_{1}<\bbb$. We will need the following definition.
\begin{definition}
    For $k\in\omega$ let $\IAP(k)$ denote the set of all increasing sequences of arithmetic progressions of length $k$, i.e. $\langle a_{m}:m\in\omega\rangle\in \IAP(k)$ if every $a_{m}$ is an arithmetic progression of length $k$ and $\max a_{m}<\min a_{m+1}$ for every $m\in\omega$.
\end{definition}
Note that a set $A\subseteq\omega$ is $\WANDER$-positive if and only if for every $k\in\omega$ there is $\langle a_{m}:m\in\omega\rangle\in \IAP(k)$ such that $\bigcup_{n}a_{n}\subseteq A$. If $a\subseteq\omega$ is an arithmetic progressions of length $k$, we will briefly write that $a$ is $k$-a.p. We will also need the following property.
\begin{definition}\label{kanpproperty}
    Assume that $k\in\omega$ and let $\mathbb{P}$ be notion of forcing. We will say that $\mathbb{P}$ has the property $(*)_{\IAP(k)}$, if for every $\{A_{N}:N\in\omega\}\subseteq\WANDER^{+}$ and every sequence of $\mathbb{P}$-names $ \langle\dot{a}_{m}:m\in\omega\rangle\in \IAP(k)$ such that $\Vdash"\forall N\in\omega$ $\forall^{\infty}m$ $\dot{a}_{m}\cap A_{N}=\emptyset"$ there is $\{\langle b^{n}_{m}:m\in\omega\rangle:n\in\omega\}\subseteq\IAP(k)$ such that
    \begin{itemize}
        \item[1)] for all $n,N\in\omega$ we have $\forall^{\infty}m$ $b^{n}_{m}\cap A_{N}=\emptyset$, 
        \item[2)] for every $\langle c_{m}:m\in\omega\rangle\in\IAP(k)$, if for every $n\in\omega$ there are infinitely many $m\in\omega$ such that $c_{m}=b^{n}_{m}$, then $\Vdash"\exists^{\infty}_{m\in\omega} c_{m}=\dot{a}_{m}"$ 
    \end{itemize}
\end{definition}
We will show that iterations of Hechler forcing has $(*)_{\IAP(k)}$-property. First we check this for Hechler forcing itself.
\begin{lemma}\label{lemmaIteratonHechler}
    Hechler forcing has the $(*)_{\IAP(k)}$-property.
\end{lemma}
\begin{proof}
    For $m\in\omega$ we will say that $\sigma\in\baire$ favors $k$-a.p. $a$ for $\dot{a}_{m}$ if there is no Hechler tree $T$ with the stem $\sigma$ that excludes $a$ for $\dot{a}_{m}$, i.e. it is not true that $T\Vdash"a\neq\dot{a}_{m}"$. Define then the function $rank_{m}:\baire\rightarrow\omega_{1}$ as follows:
    \begin{itemize}
        \item[--] $rank_{m}(\sigma)=0$ if $\sigma$ favors some $a$ for $\dot{a}_{m}$,
        \item[--] $rank_{m}(\sigma)=\alpha$ if it is not the case that $rank_{m}(\sigma)<\alpha$ and there are infinitely many $i\in\omega$ such that $rank_{m}(\sigma^{\frown}i)<\alpha$
    \end{itemize}
First we notice that $rank_{m}$ is well-defined for every $\sigma\in\baire$: if for some $\sigma$ that would not be the case, then we could inductively construct a Hechler tree $T$ with the stem $\sigma$ such that $rank_{m}$ is not defined for every $\tau\in T$. But then no extension of $T$ can decide the value of $\dot{a}_{m}$.\\
For $\sigma\in\baire$ consider the following two situations:
\begin{itemize}
    \item[1.] If $rank_{m}(\sigma)=0$ for infinitely many $m\in\omega$, we call $\sigma$ to be of type $1$. Then, let $W_{\sigma}\in[\omega]^{\omega}$ be the set of these $m$'s for which $rank_{m}(\sigma)=0$. Let $\langle b^{\sigma}_{m}:m\in W_{\sigma}\rangle\in\IAP(k)$ be such that $\sigma$ favors $b^{\sigma}_{m}$ for $\dot{a}_{m}$ for every $m\in W_{\sigma}$. Notice that for every $N\in\omega$ and $m\in W_{\sigma}\setminus N$ there is $T\in\mathbb{D}$ such that $T\Vdash"b^{\sigma}_{m}=\dot{a}_{m}$ and $\dot{a}_{m}\cap A_{N}=\emptyset"$. By absoluteness we have that $b^{\sigma}_{m}\cap A_{N}=\emptyset$.
    \item[2.] If for some $m\in\omega$ we have $rank_{m}(\sigma)=1$ we call $\sigma$ to be of type $2$. Let then $W_{\sigma}\in[\omega]^{\omega}$ be the set of these $i\in\omega$ such that $rank_{m}(\sigma^{\frown}i)=0$. Let $\langle b^{\sigma}_{i}:i\in W_{\sigma}\rangle\in \IAP(k)$ be such that for every $i\in W_{\sigma}$, $\sigma$ favors $b^{\sigma}_{i}$ for $\dot{a}_{m}$. By the similar argument as in the first case we have that $b^{\sigma}_{i}\cap A_{N}=\emptyset$ for $i\in W_{\sigma}\setminus N$. 
\end{itemize}
We claim that $\{\langle b^{\sigma}_{m}:m\in W_{\sigma}\rangle:\sigma\in\baire$ is of type $1$ or $2\}$ is the desired sequence. Let $\langle c_{m}:m\in\omega\rangle\in\IAP(k)$ be such that for any $\sigma\in\baire$ of type $1$ or $2$ we have that $c_{m}=b^{\sigma}_{m}$ for infinitely many $m$'s. We need to show that for every $T\in\mathbb{D}$ and $l\in\omega$ there is $S\leq T$ that forces $"\exists m>l$ $c_{m}=\dot{a}_{m}"$. Assume that $T\in\mathbb{D}$ and $l\in\omega$ are given. Let $\sigma=stem(T)$. We consider two cases: First, if $\sigma$ is of type $1$, let $m>l$ be such that $m\in W_{\sigma}$ and $c_{m}=b^{\sigma}_{m}$. Then, as $rank_{m}(\sigma)=0$, there is $S\leq T$ such that $S\Vdash"c_{m}=b^{\sigma}_{m}=\dot{a}_{m}"$ and we are done. Second, for some $m\in\omega$ we have $rank_{m}(\sigma)>0$. Let $\tau\in T$, $\sigma\subseteq\tau$ be such that $rank_{m}(\tau)=1$. Let $i\in W_{\tau}\setminus l$ be such that $b^{\sigma}_{i}=c_{i}$. Then there is $S\leq T|_{\sigma^{\frown}i}$ such that $S\Vdash"c_{i}=b^{\sigma}_{i}=\dot{a}_{m}"$.
\end{proof}
Next, we show that $(*)_{\IAP(k)}$-property is preserved under finite support iterations.
\begin{lemma}
    $(*)_{\IAP(k)}$-property is preserved under finite support iterations.
\end{lemma}
\begin{proof}
It is not difficult to see that $(*)_{\IAP(k)}$-property is preserved under two stage iteration and iterations of length of uncountable cofinality. Thus, we will only handle the limit stages of cofinality $\omega$.\\
Assume that $k\in\omega$ and $\{A_{N}:N\in\omega\}\subseteq\WANDER^{+}$ are given and $\langle\dot{a}_{m}:m\in\omega\rangle$ is a sequence of $\mathbb{P}_{\omega}$-names such that 
\begin{center}
    $\Vdash_{\omega}"\langle\dot{a}_{m}:m\in\omega\rangle\in\IAP(k)$ and $\forall N\in\omega$ $\forall^{\infty}_{m\in\omega}$ $\dot{a}_{m}\cap A_{N}=\emptyset"$
\end{center}
By taking interpretations of $\langle\dot{a}_{m}:m\in\omega\rangle$ in every intermediate model, for every $k\in\omega$ we find $\{\dot{q}^{k}_{m}:m\in\omega\}$ and $\langle\dot{a}^{k}_{m}:m\in\omega\rangle$ such that
\begin{center}
    $\Vdash_{k}"\{\dot{q}^{k}_{m}:m\in\omega\}\subseteq\mathbb{P}_{k}$ is decreasing and $\dot{q}^{k}_{m}\Vdash_{[k,\omega)}"\dot{a}^{k}_{m}=\dot{a}_{m}""$
\end{center}
Applying $(*)_{\IAP(k)}$-property of $\mathbb{P}_{k}$ to the $\langle\dot{a}^{k}_{m}:m\in\omega\rangle$ we find an appropriate countable collection $\{\langle b^{n,k}:m\in\omega\rangle:n\in\omega\}$. We claim that, if $\langle c_{m}:m\in\omega\rangle$ is such that for every $k,n\in\omega$ there are infinitely $m\in\omega$ such that $c_{m}=b^{n,k}_{m}$, then $\Vdash_{\omega}"\exists^{\infty}c_{m}=\dot{a}_{m}"$. Let $\langle c_{m}:m\in\omega\rangle$ be such a sequence. It will suffice to show that for every condition $p\in\mathbb{P}_{\omega}$ and $M\in\omega$ there is $q\leq p$ such that $q\Vdash_{\omega}"\exists m>M$ $ c_{m}=\dot{a}_{m}"$. As we are dealing with finite support iterations, we can find $k\in\omega$ such that $p\in\mathbb{P}_{k}$. Let $\dot{m}$ be such that $p\Vdash_{k}"\dot{m}>M$ and $c_{\dot{m}}=\dot{a}^{k}_{\dot{m}}"$. Then $q=p^{\frown}p^{k}_{\dot{m}}$ is a condition in $\mathbb{P}_{\omega}$ such that $q\leq p$ and $q\Vdash_{\omega}"c_{\dot{m}}=\dot{a}^{k}_{\dot{m}}=\dot{a}_{\dot{m}}"$ which finishes the proof.
\end{proof}

We will now show that $\aaa(\WANDER)$ is equal to $\omega_{1}$ in Hechler's model. In particular, it is not true that $\bbb\leq\aaa(\WANDER)$ holds in ZFC. We will show that by constructing an antichain in $\mathcal{P}(\omega)/\WANDER$ of size $\omega_{1}$ that is indestructible by finite support iteration of Hechler forcing of length $\omega_{2}$. To do this we will use  lemmas above, by which we know that for any $k\in\omega$ finite support iterations of Hechler forcing have the $(*)_{\IAP(k)}$-property. The followign is the main result of this section.
\begin{theorem}\label{theoremHechlerWander}
    $\aaa(\WANDER)=\omega_{1}$ holds in Hechler's model.
\end{theorem}
\begin{proof}
    For $\alpha\leq\omega_{2}$ denote by $\mathbb{D}_{\alpha}$ the finite support iteration of length $\alpha$ of Hechler forcing. We will inductively construct a collection $\mathcal{A}=\{A_{\alpha}:\alpha<\omega_{1}\}\subseteq\WANDER^{+}$ such that 
    \begin{itemize}
        \item[--] $\forall \alpha\neq\beta$ $A_{\alpha}\cap A_{\beta}\in\WANDER$,
        \item[--] $\Vdash_{\omega_{1}}"\forall \dot{A}\in\WANDER^{+}$ $\exists\alpha<\omega_{1}$ $\dot{A}\cap A_{\alpha}\in\WANDER^{+}"$.
    \end{itemize}
    It follows that $\mathcal{A}$ is a witness for $\aaa(\WANDER)=\omega_{1}$ after adding $\omega_{1}$-many Hechler reals. Before the construction of $\mathcal{A}$, we will first show that such $\mathcal{A}$ will be a witness for $\aaa(\WANDER)=\omega_{1}$ after adding $\omega_{2}$-many Hechler reals.\\
    The argument relies on complete embeddability and restricted iterations. We will only sketch the argument. For detailed explanation the reader may consult \cite{Yoryoka} or \cite{BrendleAntichains}.\\
    For $X\subseteq\omega_{2}$ we inductively define restricted iterations $\mathbb{D}_{\alpha\cap X}$, $\alpha\leq\omega_{2}$ as follows:
    on limit steps we take direct limits and on successor steps we take
    \begin{equation*}
        \mathbb{D}_{(\alpha+1)\cap X} =
        \begin{cases}
             \mathbb{D}_{(\alpha+1)\cap X}*(\dot{trivial forcing}) & \text{if } \alpha\notin X\\
             \mathbb{D}_{(\alpha+1)\cap X}*\dot{\mathbb{D}} & \text{if } \alpha\in X\\
        \end{cases}
    \end{equation*}
    We will need the following facts (for the proof see \cite{Yoryoka} section 3 or \cite{BrendleAntichains}):
    \begin{proposition}
    We have the following
    \begin{itemize}
        \item[--] $\mathbb{D}_{\omega_{2}\cap X}$ is isomorphic to $\mathbb{D}_{\beta}$ where $\beta$ is the order type of $X$,
        \item[--] $\mathbb{D}_{\alpha\cap X}$ is complete subforcing of $\mathbb{D}_{\alpha}$ for every $\alpha\leq\omega_{2}$,
        \item[--] for every $\dot{A}$ $\mathcal{D}_{\omega_{2}}$-name for infinite subset of $\omega$ there is a countable set $X\subseteq\omega_{2}$ such that $\dot{A}$ is (equivalent to) $\mathbb{D}_{\omega_{2}\cap X}$-name
    \end{itemize}
    \end{proposition}
    It follows that for any name $\dot{A}$ for an infinite subset of $\omega$, interpretation of $\dot{A}$ appears in the $\mathbb{D}_{\alpha}$-extension, where $\alpha<\omega_{1}$ is the order type of the countable set $X$. In particular, as $\mathcal{A}$ is maximal in $\mathcal{P}(\omega)\setminus\WANDER$ in $\mathbb{V}^{\mathbb{D}_{\omega_{1}}}$, $\mathcal{A}$ will be maximal in $\mathbb{V}^{\mathbb{D}_{\omega_{2}}}$.\\

    Construction of $\mathbb{D}_{\omega_{1}}$-indestructible $\mathcal{A}=\{A_{\alpha}:\alpha<\omega_{1}\}$:\\
    Using CH in the ground model enumerate as $\{(p_{\alpha},\dot{A}_{\alpha}):\alpha<\omega_{1}\}$ all pairs $(p,\dot{A})$ where $p\in\mathbb{D}_{\omega_{1}}$ and $\dot{A}$ is a $\mathbb{D}_{\omega_{1}}$-name for $\WANDER$-positive subset of $\omega$.\\
    Assume that $\{\dot{A}_{\gamma}:\gamma<\alpha\}$ has been constructed and we want to build $A_{\alpha}\in\WANDER^{+}$. Enumerate $\alpha$ as $\{\gamma_{N}:N<\omega\}$. If $p_{\alpha}\Vdash"\exists N\in\omega$ $\dot{A}\cap A_{\gamma_{N}}\in\WANDER^{+}"$ then we let $A_{\alpha}$ to be any $\WANDER$-positive set such that $A_{\alpha}\cap A_{\gamma_{N}}\in\WANDER$ for all $N\in\omega$. If  $p_{\alpha}\Vdash"\forall N\in\omega$ $\dot{A}\cap A_{\gamma_{N}}\in\WANDER"$ then for every $k\in\omega$ find $\langle\dot{a}^{k}_{m}:m\in\omega\rangle$ such that 
    \begin{center}
    $p_{\alpha}\Vdash"\langle\dot{a}^{k}_{m}:m\in\omega\rangle\in\IAP(k)$ and $\forall N\in\omega$ $\bigcup_{m\geq N}\dot{a}^{k}_{m}\subseteq\dot{A}\setminus A_{\gamma_{N}}"$
    \end{center}
    By Lemma \ref{lemmaIteratonHechler}, for every $k\in\omega$ there is an appropriate countable collection $\{\langle b^{k,n}_{m}:m\in\omega\rangle:n\in\omega\}\subseteq\IAP(k)$. Let $\omega$ be partitioned into infinite $P_{k}$'s. We will inductively construct a sequence $\langle c_{m}:m\in\omega\rangle$ of finite sets, such that:
    \begin{itemize}
        \item[$a)$] $\forall k\in\omega$ $\langle c_{m}:m\in P_{k}\rangle\in\IAP(k)$,
        \item[$b)$] $\forall k\in\omega$ $\forall n\in\omega$ $\exists^{\infty}_{m\in P_{k}}$ $c_{m}=b^{n,k}_{m}$,
        \item[$c)$] $\forall N\in\omega$ $\bigcup_{m>N}c_{m}\cap A_{\gamma_{N}}=\emptyset$
    \end{itemize}
    Notice that when such $\langle c_{m}:m\in\omega\rangle$ is constructed then the set $A_{\alpha}=\bigcup_{m}c_{m}$ is $\WANDER$-positive by the condition $a)$, for every $N\in\omega$ the intersection $A_{\alpha}\cap A_{\gamma_{N}}$ is finite by the condition $c)$ and we have $p_{\alpha}\Vdash"A_{\alpha}\cap \dot{A}\in\WANDER^{+}"$ by the condition $b)$ and the property of the sequence $\{\langle b^{k,n}_{m}:m\in\omega\rangle:n\in\omega\}$ stated in the definition \ref{kanpproperty} (by the Lemma \ref{lemmaIteratonHechler}). This finishes the construction of $\mathcal{A}$.\\
    
    Construction of the sequence $\langle c_{m}:m\in\omega\rangle$:\\
    For every $k\in\omega$ let $\psi_{k}:P_{k}\rightarrow \omega$ be such function that $\phi^{-1}_{k}(n)$ is infinite for every $n\in\omega$. Assume that $\langle c_{m}:m<M\rangle$ has been constructed and we want to construct $c_{M}$. Let $k$ be such that $M\in P_{k}$. Let $c_{M}=b^{k,\phi_{k}(M)}_{m}$ for some $m\in\omega$ such that $b^{k,\phi_{k}(M)}_{m}\cap \bigcup_{n\leq M}A_{\gamma_{n}}=\emptyset$. Finding such $m\in\omega$ is possible because of the condition $1)$ of the definition \ref{kanpproperty} and the Lemma \ref{lemmaIteratonHechler}.
\end{proof}
\begin{remark}
    Identical argument shows that $\aaa(\WANDER)=\omega_{1}<\bbb=\kappa$ after finite support iteration of Hechler forcing of regular length $\kappa$.
\end{remark}
The arguments presented above work also for linear growth ideal.
\begin{definition}
    Let $(I_{n})_{n\in\omega}$ be a partition on $\omega$ into intervals such that $|I_{n}|=2^{n}$ for every $n\in\omega$. The linear growth ideal $\mathcal{I}_{L}$ consists of those $A\subseteq\omega$ for which there is $k\in\omega$ such that for almost all $n\in\omega$ we have $|A\cap I_{n}|\leq k\cdot n$.
\end{definition}
Let us mention that $\mathcal{I}_{L}$ is a tall, $F_{\sigma}$-ideal that is above $\ED_{fin}$ in the Katětov order. It turns out that the same arguments as presented above can be used to show that $\aaa(\mathcal{I}_{L})<\bbb$ holds in Hechler's model. Due to the similarities with the case of $\WANDER$, we will only sketch the argument.
\begin{definition}
    For $k\in\omega$ let $\ILB(k)$ denote the set of all increasing $k$-linear blocks, i.e. $\langle a_{m}:m\in\omega\rangle\in \ILB(k)$ if there is an increasing $\{n_{m}:m\in\omega\}\subseteq\omega$ such that for every $m$, the set $a_{m}\subseteq I_{n_{m}}$ is of size at most $k\cdot n_{m}$.
\end{definition}
Note that $A\subseteq\omega$ is $\mathcal{I}_{L}$-positive if and only if for every $k\in\omega$ there is a sequence $\langle a_{m}:m\in\omega\rangle\in \ILB(k)$ such that $\bigcup_{m}a_{m}\subseteq A$. The forcing property related to $\ILB(k)$, that corresponds to the $(*)_{\IAP(k)}$-property, is the following.
\begin{definition}
    A forcing notion $\mathbb{P}$ has the property $(*)_{\ILB(k)}$ if for every $k\in\omega$ and every $\{A_{N}:N\in\omega\}\subseteq\mathcal{I}_{L}^{+}$ for every $ \langle\dot{a}_{m}:m\in\omega\rangle\in \ILB(k)$ sequence of $\mathbb{P}$-names such that $\Vdash"\forall N\in\omega$ $\forall^{\infty}m$ $\dot{a}_{m}\cap A_{N}=\emptyset"$ there is $\{\langle b^{n}_{m}:m\in\omega\rangle:n\in\omega\}\subseteq\ILB(k)$ such that
    \begin{itemize}
        \item[1)] for all $n,N\in\omega$ we have $\forall^{\infty}m$ $b^{n}_{N}\cap A_{N}=\emptyset$, 
        \item[2)] for every $\langle c_{m}:m\in\omega\rangle\in\ILB(k)$, if for all $n\in\omega$ there are infinitely many $m\in\omega$ such that $c_{m}=b^{n}_{m}$, then we have that $\Vdash"\exists^{\infty}_{m\in\omega} c_{m}=\dot{a}_{m}"$ 
    \end{itemize}
\end{definition}
Similarly as for $\WANDER$ one can show that 
\begin{lemma}
    Hechler forcing and its finite support iterations (of any length) have the $(*)_{\ILB(k)}$-property for every $k\in\omega$.
\end{lemma}
Using similar arguments as in the proof of Theorem \ref{theoremHechlerWander} together with the lemma above one can show that
\begin{theorem}
    $\aaa(\mathcal{J}_{L})=\omega_{1}$ in Hechler's model.
\end{theorem}
We finish with several questions regarding $\aaa(\J)$ for several other ideals. By Theorem \ref{MAINthmAntichains} we know that for Solecki ideal $\Solecki$ we have $\aaa(\Solecki)\geq\min\{\bbb,\non(\Null)\}$. If one would like to show the consistency of $\aaa(\Solecki)<\bbb$, the Hechler forcing would not work as it increases both $\bbb$ and $\non(\Null)$. However, Laver forcing and its countable support iterations preserve outer measure (in \cite{BJ} see Theorem 7.3.37 and Theorem 7.3.39) and thus keep the $\non(\Null)$ small. This motivates the following question.
\begin{question}
    Is $\aaa(\Solecki)=\omega_{1}$ or $\aaa(\ED_{fin})=\omega_{1}$ in Laver's model?
\end{question}
We have already seen that for ideal $\ED_{fin}$ and $\mathcal{I}_{L}$ the inequalities $\aaa(\ED_{fin})<\bbb$ and $\aaa(\mathcal{I}_{L})<\bbb$ are consistent. However, those two ideals belong to a wider class of $F_{\sigma}$-ideals known as fragmented ideals. These ideals were systematically investigated in \cite{HruZapletalRojas} and \cite{BrendleDiego}. It would be interesting to see that $\aaa(\J)<\bbb$ is consistent for other fragmented ideals. This motivates the following question.
\begin{question}
    What can be said about $\aaa(\J)$ where $\J$ is a fragmented ideal? Is $\aaa(\J)<\bbb$ consistent for all gradually fragmented ideals?
\end{question}
Another ideals for which we may look for more precise lower boundaries are the following.
\begin{question}
    What can be said about $\aaa(\SC)$ and $\aaa(\SPL)$?
\end{question}

\begin{acknowledgements}
I would like to thank Rafał Filipów for his hospitality during my visits in Gdańsk and discussions regarding the splitting ideal. I would like to thank Adam Kwela for suggesting the reduction $\conv\leq_{K}\SPL$. I would also like to thank J\"{o}rg Brendle for discussions about antichain numbers and showing me the proof of $\aaa(\ED_{fin})=\omega_{1}$ in Hechler's model, which led to the consistency of $\aaa(\WANDER)<\bbb$.
\end{acknowledgements}

\bibliographystyle{alphadin} 
\bibliography{bib.bib}

@article {FarkSoukup,
    AUTHOR = {Farkas, Barnab\'as and Soukup, Lajos},
     TITLE = {More on cardinal invariants of analytic {$P$}-ideals},
   JOURNAL = {Comment. Math. Univ. Carolin.},
  FJOURNAL = {Commentationes Mathematicae Universitatis Carolinae},
    VOLUME = {50},
      YEAR = {2009},
    NUMBER = {2},
     PAGES = {281--295},
      ISSN = {0010-2628,1213-7243},
   MRCLASS = {03E17 (03E05)},
  MRNUMBER = {2537837},
MRREVIEWER = {Teruyuki\ Yorioka},
}

@book {BJ,
    AUTHOR = {Bartoszy\'nski, Tomek and Judah, Haim},
     TITLE = {Set theory},
      NOTE = {On the structure of the real line},
 PUBLISHER = {A K Peters, Ltd., Wellesley, MA},
      YEAR = {1995},
     PAGES = {xii+546},
      ISBN = {1-56881-044-X},
   MRCLASS = {03-02 (03Exx)},
  MRNUMBER = {1350295},
MRREVIEWER = {Eva\ Coplakova},
}

@incollection {Hrus,
    AUTHOR = {Hru\v s\'ak, Michael},
     TITLE = {Combinatorics of filters and ideals},
 BOOKTITLE = {Set theory and its applications},
    SERIES = {Contemp. Math.},
    VOLUME = {533},
     PAGES = {29--69},
 PUBLISHER = {Amer. Math. Soc., Providence, RI},
      YEAR = {2011},
      ISBN = {978-0-8218-4812-8},
   MRCLASS = {03E15 (03E05 03E17 03E35)},
  MRNUMBER = {2777744},
MRREVIEWER = {Jan\ Kraszewski},
       DOI = {10.1090/conm/533/10503},
       URL = {https://doi.org/10.1090/conm/533/10503},
}

@article {Stevo,
    AUTHOR = {Todor\v cevi\'c, Stevo},
     TITLE = {Analytic gaps},
   JOURNAL = {Fund. Math.},
  FJOURNAL = {Fundamenta Mathematicae},
    VOLUME = {150},
      YEAR = {1996},
    NUMBER = {1},
     PAGES = {55--66},
      ISSN = {0016-2736,1730-6329},
   MRCLASS = {03E15 (03E05)},
  MRNUMBER = {1387957},
MRREVIEWER = {Wim\ Veldman},
       DOI = {10.4064/fm-150-1-55-66},
       URL = {https://doi.org/10.4064/fm-150-1-55-66},
}

@article {Hernand,
    AUTHOR = {Hern\'{a}ndez-Hern\'{a}ndez, Fernando and Hru\v{s}\'{a}k,
              Michael},
     TITLE = {Cardinal invariants of analytic {$P$}-ideals},
   JOURNAL = {Canad. J. Math.},
  FJOURNAL = {Canadian Journal of Mathematics. Journal Canadien de
              Math\'{e}matiques},
    VOLUME = {59},
      YEAR = {2007},
    NUMBER = {3},
     PAGES = {575--595},
      ISSN = {0008-414X,1496-4279},
   MRCLASS = {03E17 (03E35 03E40)},
  MRNUMBER = {2319159},
MRREVIEWER = {J\"{o}rg\ D.\ Brendle},
       DOI = {10.4153/CJM-2007-024-8},
       URL = {https://doi.org/10.4153/CJM-2007-024-8},
}

@article {KwelaUnboring,
    AUTHOR = {Kwela, Adam},
     TITLE = {Unboring ideals},
   JOURNAL = {Fund. Math.},
  FJOURNAL = {Fundamenta Mathematicae},
    VOLUME = {261},
      YEAR = {2023},
    NUMBER = {3},
     PAGES = {235--272},
      ISSN = {0016-2736,1730-6329},
   MRCLASS = {03E05 (03E15 03E35 26A03 40A05 40A35 54A20 54H05)},
  MRNUMBER = {4584767},
MRREVIEWER = {Klaas\ Pieter\ Hart},
       DOI = {10.4064/fm44-2-2023},
       URL = {https://doi.org/10.4064/fm44-2-2023},
}

@article {Meza,
    AUTHOR = {David Meza-Alcántara},
     TITLE = {Ideals and filtres on countable sets},
   JOURNAL = {Phd thesis},
  FJOURNAL = {},
    VOLUME = {},
      YEAR = {2009},
    NUMBER = {},
     PAGES = {},
      ISSN = {},
   MRCLASS = {},
  MRNUMBER = {},
MRREVIEWER = {},
       DOI = {},
       URL = {},}

@article {SabZapl,
    AUTHOR = {Sabok, Marcin and Zapletal, Jind\v rich},
     TITLE = {Forcing properties of ideals of closed sets},
   JOURNAL = {J. Symbolic Logic},
  FJOURNAL = {The Journal of Symbolic Logic},
    VOLUME = {76},
      YEAR = {2011},
    NUMBER = {3},
     PAGES = {1075--1095},
      ISSN = {0022-4812,1943-5886},
   MRCLASS = {03E40 (03E15 26A21 28A05 54H05)},
  MRNUMBER = {2849260},
MRREVIEWER = {J\"org\ D.\ Brendle},
       DOI = {10.2178/jsl/1309952535},
       URL = {https://doi.org/10.2178/jsl/1309952535},
}

@article {RaghavanShelah,
    AUTHOR = {Raghavan, Dilip and Shelah, Saharon},
     TITLE = {Two inequalities between cardinal invariants},
   JOURNAL = {Fund. Math.},
  FJOURNAL = {Fundamenta Mathematicae},
    VOLUME = {237},
      YEAR = {2017},
    NUMBER = {2},
     PAGES = {187--200},
      ISSN = {0016-2736,1730-6329},
   MRCLASS = {03E17 (03E05 03E20 03E55)},
  MRNUMBER = {3615051},
MRREVIEWER = {Piotr\ Borodulin-Nadzieja},
       DOI = {10.4064/fm253-7-2016},
       URL = {https://doi.org/10.4064/fm253-7-2016},
}

@article {Steprans,
    AUTHOR = {Stepr\=ans, Juris},
     TITLE = {The almost disjointness cardinal invariant in the quotient
              algebra of the rationals modulo the nowhere dense subsets},
   JOURNAL = {Real Anal. Exchange},
  FJOURNAL = {Real Analysis Exchange},
    VOLUME = {27},
      YEAR = {2001/02},
    NUMBER = {2},
     PAGES = {795--800},
      ISSN = {0147-1937,1930-1219},
   MRCLASS = {03E17 (03E35)},
  MRNUMBER = {1923169},
       DOI = {10.14321/realanalexch.27.2.0795},
       URL = {https://doi.org/10.14321/realanalexch.27.2.0795},
}

@article {BrendleAntichains ,
    AUTHOR = {Brendle, Jörg},
     TITLE = {Cardinal invariants of analytic quotients},
   JOURNAL = {},
  FJOURNAL = {},
    VOLUME = {},
      YEAR = {2009},
    NUMBER = {},
     PAGES = {},
      ISSN = {},
   MRCLASS = {)},
  MRNUMBER = {},
       DOI = {https://www.logic.univie.ac.at/2009/esi/pdf/brendle.pdf},
       URL = {},}

@article {FarkasZdomskyy,
    AUTHOR = {Farkas, Barnab\'as and Zdomskyy, Lyubomyr},
     TITLE = {Ways of destruction},
   JOURNAL = {J. Symb. Log.},
  FJOURNAL = {The Journal of Symbolic Logic},
    VOLUME = {87},
      YEAR = {2022},
    NUMBER = {3},
     PAGES = {938--966},
      ISSN = {0022-4812,1943-5886},
   MRCLASS = {03E05 (03E15 03E17 03E35)},
  MRNUMBER = {4472520},
MRREVIEWER = {Wei\ Wang},
       DOI = {10.1017/jsl.2021.84},
       URL = {https://doi.org/10.1017/jsl.2021.84},
}

@article {SpHein,
    AUTHOR = {Hein, Paul and Spinas, Otmar},
     TITLE = {Antichains of perfect and splitting trees},
   JOURNAL = {Arch. Math. Logic},
  FJOURNAL = {Archive for Mathematical Logic},
    VOLUME = {59},
      YEAR = {2020},
    NUMBER = {3-4},
     PAGES = {367--388},
      ISSN = {0933-5846,1432-0665},
   MRCLASS = {03E35 (03E17)},
  MRNUMBER = {4081065},
MRREVIEWER = {Piotr\ Borodulin-Nadzieja},
       DOI = {10.1007/s00153-019-00694-7},
       URL = {https://doi.org/10.1007/s00153-019-00694-7},
}

@article {CIESGAPPOYAMAZOEMARTINEZ,
    AUTHOR = {Aleksander Cieślak and Takehiko Gappo and Arturo Martínez-Celis and Takashi Yamazoe},
     TITLE = {Cardinal invariants of idealized Miller null sets},
   JOURNAL = {},
  FJOURNAL = {},
    VOLUME = {},
      YEAR = {2026},
    NUMBER = {},
     PAGES = {},
      ISSN = {},
   MRCLASS = {},
  MRNUMBER = {},
       DOI = {https://doi.org/10.48550/arXiv.2601.07428},
       URL = {},}

@article{CIESMARTINEZ, 
    title={On ideals related to Laver and Miller trees}, 
    DOI={10.1017/jsl.2025.10130}, 
    journal={The Journal of Symbolic Logic}, 
    author={Cie{\'s}lak, Aleksander and Mart{\'\i}nez-Celis, Arturo}, 
    year={2025}, 
    pages={1-–19}
}

@article {BS99,
    AUTHOR = {Brendle, J\"org and Shelah, Saharon},
     TITLE = {Ultrafilters on {$\omega$}---their ideals and their cardinal
              characteristics},
   JOURNAL = {Trans. Amer. Math. Soc.},
  FJOURNAL = {Transactions of the American Mathematical Society},
    VOLUME = {351},
      YEAR = {1999},
    NUMBER = {7},
     PAGES = {2643--2674},
      ISSN = {0002-9947,1088-6850},
   MRCLASS = {03E05 (03E35)},
  MRNUMBER = {1686797},
MRREVIEWER = {Pierre\ Matet},
       DOI = {10.1090/S0002-9947-99-02257-6},
       URL = {https://doi.org/10.1090/S0002-9947-99-02257-6},
}

@article {FilipowKwela,
    AUTHOR = {Filip\'ow, Rafa\l{} and Kwela, Adam},
     TITLE = {Spaces not distinguishing ideal pointwise and {\it {$\sigma
              $}}-uniform convergence},
   JOURNAL = {Ann. Pure Appl. Logic},
  FJOURNAL = {Annals of Pure and Applied Logic},
    VOLUME = {176},
      YEAR = {2025},
    NUMBER = {9},
     PAGES = {Paper No. 103609, 27},
      ISSN = {0168-0072,1873-2461},
   MRCLASS = {54C30 (03E17 03E35 40A30 40A35 54A20)},
  MRNUMBER = {4905423},
MRREVIEWER = {Binod\ Chandra\ Tripathy},
       DOI = {10.1016/j.apal.2025.103609},
       URL = {https://doi.org/10.1016/j.apal.2025.103609},
}

@article {HrusBalcar,
    AUTHOR = {Balcar, B. and Hern\'andez-Hern\'andez, F. and Hru\v s\'ak,
              M.},
     TITLE = {Combinatorics of dense subsets of the rationals},
   JOURNAL = {Fund. Math.},
  FJOURNAL = {Fundamenta Mathematicae},
    VOLUME = {183},
      YEAR = {2004},
    NUMBER = {1},
     PAGES = {59--80},
      ISSN = {0016-2736,1730-6329},
   MRCLASS = {03E17 (03E35 06E15)},
  MRNUMBER = {2098150},
MRREVIEWER = {Andrzej\ Ros\l anowski},
       DOI = {10.4064/fm183-1-4},
       URL = {https://doi.org/10.4064/fm183-1-4},
}

@article {FilipowNice,
    AUTHOR = {Filip\'ow, Rafa\l},
     TITLE = {The reaping and splitting numbers of nice ideals},
   JOURNAL = {Colloq. Math.},
  FJOURNAL = {Colloquium Mathematicum},
    VOLUME = {134},
      YEAR = {2014},
    NUMBER = {2},
     PAGES = {179--192},
      ISSN = {0010-1354,1730-6302},
   MRCLASS = {03E17 (03E50 40A35)},
  MRNUMBER = {3194404},
MRREVIEWER = {Shimon\ Garti},
       DOI = {10.4064/cm134-2-3},
       URL = {https://doi.org/10.4064/cm134-2-3},
}

@article {BrendleDiego,
    AUTHOR = {Brendle, J\"org and Mej\'ia, Diego Alejandro},
     TITLE = {Rothberger gaps in fragmented ideals},
   JOURNAL = {Fund. Math.},
  FJOURNAL = {Fundamenta Mathematicae},
    VOLUME = {227},
      YEAR = {2014},
    NUMBER = {1},
     PAGES = {35--68},
      ISSN = {0016-2736,1730-6329},
   MRCLASS = {03E17 (03E05 03E15 03E35)},
  MRNUMBER = {3247032},
MRREVIEWER = {Jan\ Kraszewski},
       DOI = {10.4064/fm227-1-4},
       URL = {https://doi.org/10.4064/fm227-1-4},
}

@article {Mazur,
    AUTHOR = {Mazur, Krzysztof},
     TITLE = {{$F_\sigma$}-ideals and {$\omega_1\omega_1^*$}-gaps in the
              {B}oolean algebras {$P(\omega)/I$}},
   JOURNAL = {Fund. Math.},
  FJOURNAL = {Polska Akademia Nauk. Fundamenta Mathematicae},
    VOLUME = {138},
      YEAR = {1991},
    NUMBER = {2},
     PAGES = {103--111},
      ISSN = {0016-2736,1730-6329},
   MRCLASS = {06E05 (03G05)},
  MRNUMBER = {1124539},
MRREVIEWER = {Judith\ Roitman},
       DOI = {10.4064/fm-138-2-103-111},
       URL = {https://doi.org/10.4064/fm-138-2-103-111},
}

@article {SoleckiExh,
    AUTHOR = {Solecki, S\l awomir},
     TITLE = {Analytic ideals and their applications},
   JOURNAL = {Ann. Pure Appl. Logic},
  FJOURNAL = {Annals of Pure and Applied Logic},
    VOLUME = {99},
      YEAR = {1999},
    NUMBER = {1-3},
     PAGES = {51--72},
      ISSN = {0168-0072,1873-2461},
   MRCLASS = {03E15 (28A12)},
  MRNUMBER = {1708146},
MRREVIEWER = {Jakub\ Jasi\'nski},
       DOI = {10.1016/S0168-0072(98)00051-7},
       URL = {https://doi.org/10.1016/S0168-0072(98)00051-7},
}

@article {Raghavan,
    AUTHOR = {Raghavan, Dilip},
     TITLE = {The density zero ideal and the splitting number},
   JOURNAL = {Ann. Pure Appl. Logic},
  FJOURNAL = {Annals of Pure and Applied Logic},
    VOLUME = {171},
      YEAR = {2020},
    NUMBER = {7},
     PAGES = {102807, 15},
      ISSN = {0168-0072,1873-2461},
   MRCLASS = {03E17 (03E05 03E20 03E55)},
  MRNUMBER = {4099835},
MRREVIEWER = {Damian\ Sobota},
       DOI = {10.1016/j.apal.2020.102807},
       URL = {https://doi.org/10.1016/j.apal.2020.102807},
}

@article {Mathias,
    AUTHOR = {Mathias, A. R. D.},
     TITLE = {Happy families},
   JOURNAL = {Ann. Math. Logic},
  FJOURNAL = {Annals of Mathematical Logic},
    VOLUME = {12},
      YEAR = {1977},
    NUMBER = {1},
     PAGES = {59--111},
      ISSN = {0003-4843},
   MRCLASS = {04A20 (02K05 02K30 04A15)},
  MRNUMBER = {491197},
MRREVIEWER = {James\ Baumgartner},
       DOI = {10.1016/0003-4843(77)90006-7},
       URL = {https://doi.org/10.1016/0003-4843(77)90006-7},
}

@article {HrusakZapletal,
    AUTHOR = {Hru\v s\'ak, Michael and Zapletal, Jind\v rich},
     TITLE = {Forcing with quotients},
   JOURNAL = {Arch. Math. Logic},
  FJOURNAL = {Archive for Mathematical Logic},
    VOLUME = {47},
      YEAR = {2008},
    NUMBER = {7-8},
     PAGES = {719--739},
      ISSN = {0933-5846,1432-0665},
   MRCLASS = {03E40 (03E15 03E17)},
  MRNUMBER = {2448955},
MRREVIEWER = {J\"org\ D.\ Brendle},
       DOI = {10.1007/s00153-008-0104-4},
       URL = {https://doi.org/10.1007/s00153-008-0104-4},
}

@article {Yoryoka,
    AUTHOR = {Yorioka, Teruyuki},
     TITLE = {Forcings with the countable chain condition and the covering
              number of the {M}arczewski ideal},
   JOURNAL = {Arch. Math. Logic},
  FJOURNAL = {Archive for Mathematical Logic},
    VOLUME = {42},
      YEAR = {2003},
    NUMBER = {7},
     PAGES = {695--710},
      ISSN = {0933-5846,1432-0665},
   MRCLASS = {03E05 (03E17 03E35)},
  MRNUMBER = {2015095},
MRREVIEWER = {Arnold\ W.\ Miller},
       DOI = {10.1007/s00153-003-0174-2},
       URL = {https://doi.org/10.1007/s00153-003-0174-2},
}

@article {HruZapletalRojas,
    AUTHOR = {Hru\v s\'ak, Michael and Rojas-Rebolledo, Diego and Zapletal,
              Jind\v rich},
     TITLE = {Cofinalities of {B}orel ideals},
   JOURNAL = {MLQ Math. Log. Q.},
  FJOURNAL = {MLQ. Mathematical Logic Quarterly},
    VOLUME = {60},
      YEAR = {2014},
    NUMBER = {1-2},
     PAGES = {31--39},
      ISSN = {0942-5616,1521-3870},
   MRCLASS = {03E17 (03E15)},
  MRNUMBER = {3171606},
       DOI = {10.1002/malq.201200079},
       URL = {https://doi.org/10.1002/malq.201200079},
}

@article {BrendleGuzHruRaghavan,
    AUTHOR = {Jörg Brendle and Osvaldo Guzmán-González and Michael Hrušák and Dilip Raghavan},
     TITLE = {Combinatorial properties of MAD families},
   JOURNAL = {Canadian Journal of Mathematics},
  FJOURNAL = {Canadian Journal of Mathematics},
    VOLUME = {},
      YEAR = {2025},
    NUMBER = {},
     PAGES = {1-69},
      ISSN = {},
   MRCLASS = {},
  MRNUMBER = {},
MRREVIEWER = {},
       DOI = {10.4153/S0008414X25101879},
       URL = {},
       }

@article {GuzmanPhd,
    AUTHOR = {Guzmán, Osvaldo},
     TITLE = {PhD Thesis: P-points, mad families and cardinal invariants},
   JOURNAL = {PhD Thesis},
  FJOURNAL = {PhD Thesis},
    VOLUME = {},
      YEAR = {2017},
    NUMBER = {},
     PAGES = {},
      ISSN = {},
   MRCLASS = {},
  MRNUMBER = {},
MRREVIEWER = {},
       DOI = {},
       URL = {https://www.matmor.unam.mx/~oguzman/Tesis%20doctorado.pdf},
}

@article {BrendleCastro,
    AUTHOR = {Jörg Brendle and Yhon Jairo Castro and Michael Hrušák and Diego Alejandro Mejía},
     TITLE = {},
   JOURNAL = {in preparation},
  FJOURNAL = {},
    VOLUME = {},
      YEAR = {2026},
    NUMBER = {},
     PAGES = {},
      ISSN = {},
   MRCLASS = {},
  MRNUMBER = {},
MRREVIEWER = {},
       DOI = {},
       URL = {},
}

@article {Killing_Ideals_Softly,
    AUTHOR = {Aleksander Cieślak and Barnabás Farkas and Lyubomyr Zdomskyy},
     TITLE = {Killing ideals softly},
   JOURNAL = {In preparation},
  FJOURNAL = {},
    VOLUME = {},
      YEAR = {2026},
    NUMBER = {},
     PAGES = {},
      ISSN = {},
   MRCLASS = {},
  MRNUMBER = {},
MRREVIEWER = {},
       DOI = {},
       URL = {},
}

@article {HrusakMeza,
    AUTHOR = {Hru\v s\'ak, Michael and Meza-Alc\'antara, David and Minami,
              Hiroaki},
     TITLE = {Pair-splitting, pair-reaping and cardinal invariants of
              {$F_\sigma$}-ideals},
   JOURNAL = {J. Symbolic Logic},
  FJOURNAL = {The Journal of Symbolic Logic},
    VOLUME = {75},
      YEAR = {2010},
    NUMBER = {2},
     PAGES = {661--677},
      ISSN = {0022-4812,1943-5886},
   MRCLASS = {03E15 (03E05 03E17 03E35)},
  MRNUMBER = {2648159},
MRREVIEWER = {T.\ Thrivikraman},
       DOI = {10.2178/jsl/1268917498},
       URL = {https://doi.org/10.2178/jsl/1268917498},
}

\end{document}